\documentclass[12pt , a4j]{amsart}
\usepackage{ifpdf}
\usepackage{amsmath, amsthm , amssymb}
\usepackage[abbrev]{amsrefs}
\usepackage[dvipdfm]{graphicx}
\usepackage{color}
\usepackage{comment}
\usepackage{float}
\usepackage{amscd}
\usepackage{stmaryrd}
\usepackage{subcaption}
\usepackage{latexsym}
\usepackage[top=17truemm,bottom=18truemm,left=29truemm,right=29truemm]{geometry}

\theoremstyle{definition}
\newtheorem{thm}{Theorem}

\newtheorem{lem}[thm]{Lemma}
\newtheorem{cor}[thm]{Corollary}
\newtheorem{dfn}[thm]{Definition}

\newtheorem{rem}[thm]{Remark}

\begin{document}

\title{Vertical 3-manifolds in simplified $(2, 0)$-trisections of 4-manifolds}
\author{Nobutaka Asano}
\address{Mathematical Institute, Tohoku University, Sendai, 980-8578, Japan}
\email{nobutaka.asano.r4@dc.tohoku.ac.jp}
%\\footnote[0]{This work is supported by MEXT, Grant-in-Aid for Young Scientists (B) (No. 22740032).}
%\keywords{}
%\subjclass{Primary: 57M50, Secondary: 55R25, 32S30, 57R17}
%\subjclass[2010]{Primary 57R45; Secondary 58C27, 14B05}
%\date{11 August, 2006.}

\begin{abstract}
We classify the $3$-manifolds obtained as the preimages of arcs on the plane for simplified $(2, 0)$-trisection maps, which we call vertical $3$-manifolds. Such a $3$-manifold is a connected sum of a $6$-tuple of vertical $3$-manifolds over specific $6$ arcs. Consequently, we show that each of the $6$-tuples determines the source $4$-manifold uniquely up to orientation reversing diffeomorphisms. We also show that, in contrast to the fact that summands of vertical $3$-manifolds of simplified $(2, 0)$-trisection maps are lens spaces, there exist infinitely many simplified $(2, 0)$-$4$-section maps that admit hyperbolic vertical $3$-manifolds.
\end{abstract}
\maketitle

\section{Introduction}
A trisection is a decomposition of a closed orientable smooth $4$-manifold into three $4$-dimensional handlebodies introduced by Gay and Kirby \cite{Kirby}. They proved that any closed orientable smooth $4$-manifold has a trisection. While the trisection has a strong meaning as a $4$-dimensional analog of the Heegaard splitting of $3$-manifolds, it is also deeply related to the study of stable maps as homotopy-deformations of stable maps are used in their proof. 
The singular value set of a stable map of a trisection is the union of immersed circles with cusps as in Figure~\ref{tri}, where the singular value set in the white boxes consists of immersed curves with only normal double points and without cusps and radial tangencies. This stable map is called a trisection map.
The fiber $\Sigma_g$ over the center point $p$ in the figure is a closed orientable surface of genus $g$, and this has the highest genus among all regular fibers of the trisection map. Set $k$ to be the number of simple closed curves without cusps in the singular value set of the trisection map. In this case, the trisection is called a $(g, k)$-trisection.
The vanishing cycles of indefinite folds of the trisection map can be represented by simple closed curves on the center fiber $\Sigma_g$. The vanishing cycles of a trisection determines the source $4$-manifold up to diffeomorphism. The surface $\Sigma_g$ with these simple closed curves is called a trisection diagram.
To find the usage of trisections, there are several studies of constructing trisections of given $4$-manifolds \cite{koenig, meier, meier-cole}. It is also used for studies of surfaces embedded in $4$-manifolds \cite{meier-zupan, cole}. Classifications of $4$-manifolds admitting trisection maps for $g=1$ is easy, and that for $g=2$ had been done by Meier and Zupan in \cite{MZ}. The classification for $g \geq 3$ is still difficult.%However, it is still difficult to understand essential usage of trisecitons in the study of $4$-manifolds. To see the concrete relation between trisections and $4$-manifolds. Meier anf Zupan classifications the case of genus $g = 2$.

\begin{figure}[htbp]
\begin{center}
\includegraphics[clip, width=7cm, bb=152 424 441 713]{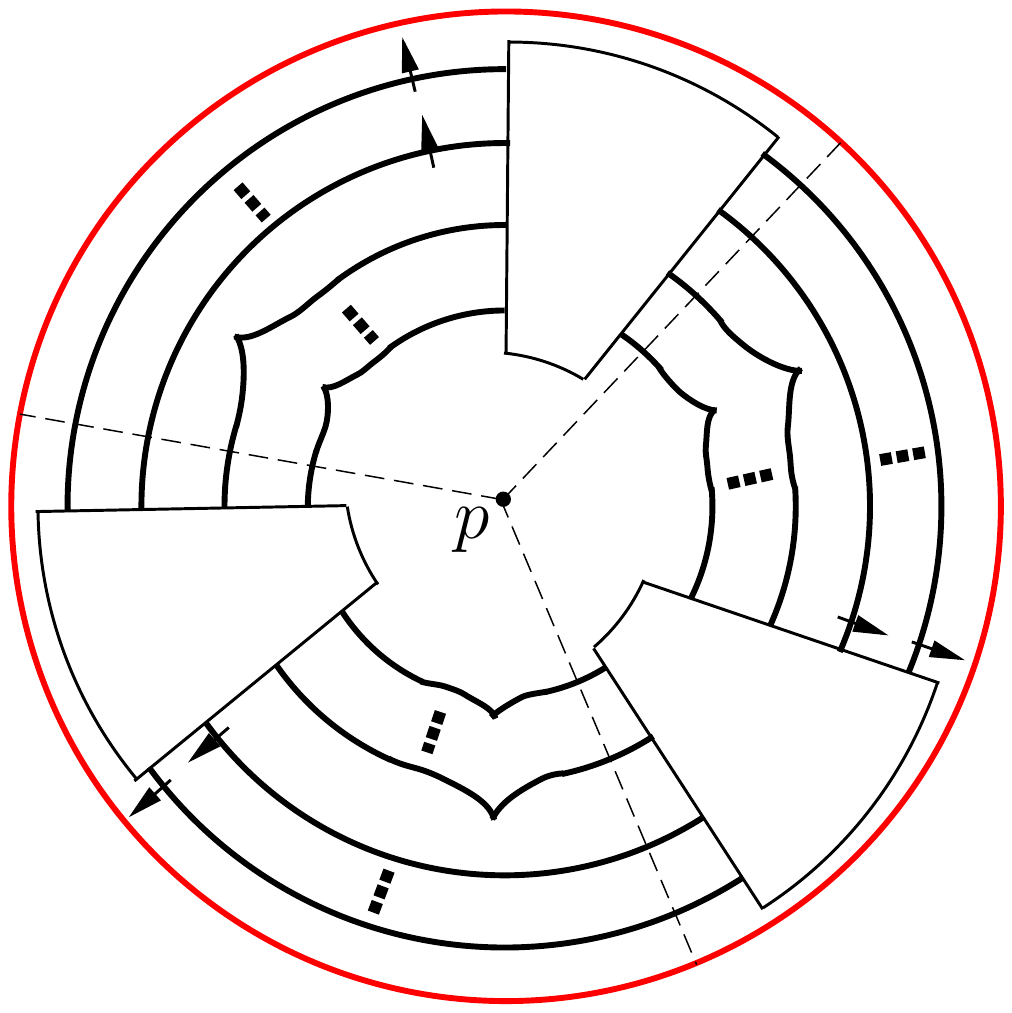}
\caption{A singular value set of a trisection map.}
\label{tri}
\end{center}
\end{figure}

On the other hand, it is a long-standing problem in the study of topology of mappings to understand the source manifold of a smooth map from the information of its singularities. In 2009, Baykur proved that any closed $4$-manifold admits a smooth map to the $2$-sphere with only Lefschetz singularities and a circle of indefinite folds \cite{Baykur}. Recently, Baykur and Saeki proved that any closed $4$-manifold admits a stable map to the plane whose singular value set is as in Figure~\ref{tri} with the singular value set in the white boxes consisting of embedded curves without cusps and radial tangencies. This gives a special case of a trisection, called a simplified $(g, k)$-trisection. They proved that any closed $4$-manifold admitting a $(2, k)$-trisection has a simplified $(2, k)$-trisection. Later, Hayano gave an alternative short proof of this result by observing monodromy diffeomorphisms generated by the Dehn twists along the vanishing cycles \cite{hayano}.

%The advantage using trisection being that curves on a surface are simpler than the framed links which serve as attaching circles of $2$-handles. For example, constructing $4$-dimensional new invariants \cite{}, \cite{},  describing geometric structures on a $4$-manifold \cite{}... 

The aim of this paper is to understand relation between trisection maps and the source $4$-manifolds from the information of $3$-manifolds obtained as the preimages of arcs on $\mathbf{R}^2$. In the study of topology of mappings, Kobayashi studied a stable map whose singular value set consists of two concentric circles, the outer one is the image of definite folds and the inner one is the image of indefinite folds and cusps, where cusps are outward~\cite{kobayashi2}. In his further study \cite{kobayashi}, he constructed an infinite number of stable maps on each of $S^2 \times S^2$ and $\mathbf{CP}^2 \# \overline{\mathbf{CP}^2}$ that have the same singular value set but that are not right-left equivalent, using what he called four cusped fans. This theorem is proved by observing the $3$-manifolds obtained as the preimage of an arc in the target space. 

%We construct two simplified $(2, 0)$-trisection maps on $\mathbf{CP}^2 \# \overline{\mathbf{CP}^2}$. These mappings are not right-left equivalent each other. 

%Saeki-yamamamoto developed an algorithm for describing Heegaard surfaces using a stable map from a 3-dimensional manifold to a real 2-dimensional plane

In this paper, we classify the $3$-manifolds obtained as the preimages of arcs on $\mathbf{R}^2$ for 
a simplified $(2, 0)$-trisection map, called vertical $3$-manifolds, and study if they determine the source $4$-manifold. The first result is the classification of the vertical $3$-manifolds. 

\begin{thm}\label{preimage}
Vertical manifolds of simplified $(2, 0)$-trisection maps are  
\begin{align*}
&{\#}^{\ell + \epsilon_1, \ell}L(k^2, k-1){\#}^{m, m+ \epsilon_2} L(k \pm 1, \pm1){\#}^n S^1 \times S^2\quad (k \in \mathbf{Z})\\
&{\#}^{\ell + \epsilon_1, \ell}L(9, 2){\#}^{m, m+ \epsilon_2} L(k, 1){\#}^n S^1 \times S^2\quad (k \in \{ 2, 5\})\\
&{\#}^{\ell + \epsilon_1, \ell}L(4, 1){\#}^{m, m+ \epsilon_2} L(k, 1){\#}^n S^1 \times S^2\quad (k \in \{ 3, 5\}).
\end{align*}
Here, $\ell, m, n \in \mathbf{Z}_{\geq 0},\epsilon_i \in \{0, 1\}$ and ${\#}^{\ell, \ell'}M$ means the connected sum of $\ell$ copies of an oriented $3$-manifold $M$ and $\ell'$ copies of the mirror image of $M$.
\end{thm}
To prove the theorem, we classify the $6$-tuples of vertical $3$-manifolds over the $6$ arcs in Figure~\ref{verTripleArc}. The vertical $3$-manifolds in Theorem~\ref{preimage} are obtained as their connected sums. We use Hayano's argument used in \cite{hayano} and study the positions of vanishing cycles of simplified  $(2, 0)$-trisection maps.

\begin{figure}[htbp]
\begin{center}
\includegraphics[clip, width=10cm, bb=129 533 518 704]{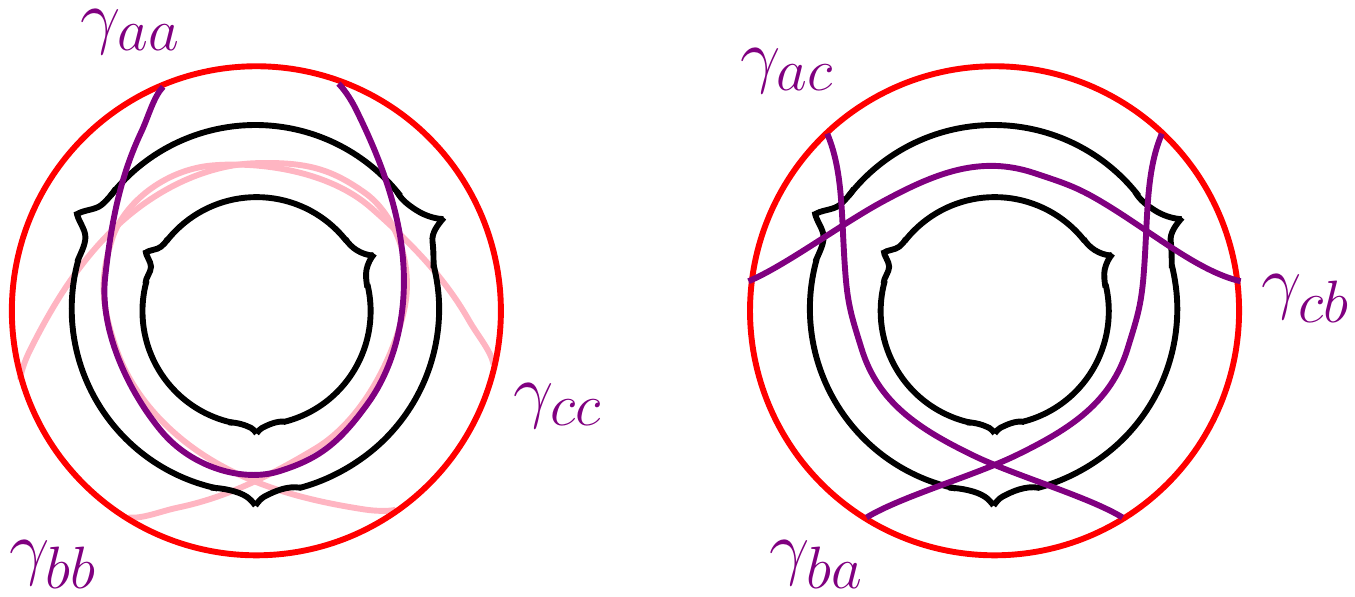}
\caption{The $6$-tuple of arcs.}
\label{verTripleArc}
\end{center}
\end{figure}

Using the information obtained in the proof of Theorem~\ref{preimage}, we determine the source $4$-manifolds from the positions of vanishing cycles. In consequence, we have the following corollary.

\begin{cor}\label{determining4MfdC}
The $6$-tuple determines the $4$-manifold unless it is  $\begin{pmatrix}
S^1 \times S^2 & S^3 & S^3\\
S^3 & L(2, 1) & S^3
\end{pmatrix}$. In this exceptional case, the $4$-manifold is determined up to orientation reversing diffeomorphisms.
\end{cor}

The detailed correspondence between $6$-tuples and $4$-manifolds can be found in Theorem~\ref{determining4Mfd}.
%%%%%%%%%%%%%%%%%%%%%%%%%%%%%
%%%%%%%%%%%%%%%%%%%%%%%%%%%%%
\begin{comment}
\begin{thm}
Let $f:X\to\mathbf{R}^2$ be a simplified $(2, 0)$-trisection map. 
\begin{itemize}
\item If $X$ contains $L(n^2, n - 1)$ with $|n| \geq 4$ and $n$ being odd as a vertical manifold, then $X$ is $S^2 \times S^2$.
\item If $X$ contains $L(n^2, n - 1)$ with $|n| \geq 4$ and $n$ being even or $L(4, 1) \# L(3, 1)$ as a vertical manifold, then $X$ is $\mathbf{CP}^2 \# \overline{\mathbf{CP}^2}$.
\item If $X$ contains $L(9, 2)$, then $X$ is either $S^2 \times S^2$ or $\mathbf{CP}^2 \# \mathbf{CP}^2$
\item If $X$ contains $L(5, 1) \# L(4, 1)$, then $X$ is $\mathbf{CP}^2 \# \mathbf{CP}^2$.
\end{itemize}
\end{thm}
\end{comment}
%%%%%%%%%%%%%%%%%%%%%%%%%%%%%
%%%%%%%%%%%%%%%%%%%%%%%%%%%%%%\begin{thm}
%$\mathbf{CP}^2 \# \overline{\mathbf{CP}^2}$ admits two simplified $(2, 0)$-trisection maps with the same vertical manifolds but not to be right-left equivalent.
%\end{thm}

From Theorem \ref{preimage}, we can see that any summand of a vertical $3$-manifold of a simplified $(2, 0)$-trisection is a lens space. On the other hand, if we consider $4$-manifolds with four sections, named $4$-sections (cf. \cite{IN}), we can obtain infinitely many different hyperbolic $3$-manifolds as vertical $3$-manifolds.

\begin{thm}\label{hyperbolic}
Suppose $X = \#^2 S^2 \times S^2$ or $\#^2\mathbf{CP}^2 \#^2 \overline{\mathbf{CP}^2}$. Then there exists a family of simplified $(2, 0)$-$4$-section maps % $\{ f_i : X \to \mathbf{R}^2 \}_{i \in \mathbf{N}}$ 
such that the vertical $3$-manifolds %$\{ M_i\}$ 
over the arc $\omega$ in Figure~\ref{2_4_gon_multi} are hyperbolic and mutually non-diffeomorphic.
\end{thm}

\begin{figure}[htbp]
\begin{center}
\includegraphics[width=6cm, bb=209 538 384 713]{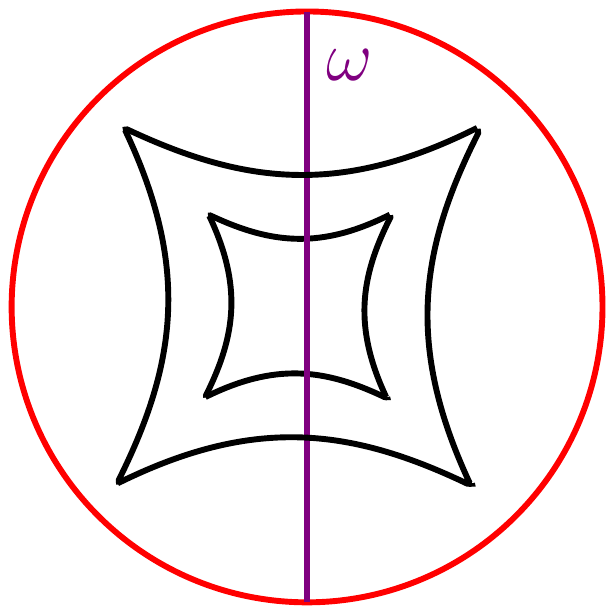}
\caption{The singular value set of a $(2, 0)$-$4$-section map.}%the vertical hyperbolic $3$-manifold $f^{-1}_i(\omega)$ and the singular value of $f_i$.}
\label{2_4_gon_multi}
\end{center}
\end{figure}

To prove the theorem, we use a handle decomposition of the $4$-manifold induced by the $4$-section map. The infinite sequence of hyperbolic $3$-manifolds is given as a sequence of surgered manifolds along hyperbolic $2$-bridge links.

The paper is organized as follows: In Section 2, we give the definition of simplified $(g, k)$-trisection maps and introduce some properties between mapping class groups and vanishing cycles of simplified $(2, 0)$-trisection maps used by Hayano in \cite{hayano}.
In Section 3, we introduce the $6$-tuples of vertical $3$-manifolds and give their classification. Theorem~\ref{preimage} is proved in Section 4. In Section 5, we determine the $4$-manifold for each $6$-tuple. Corollary ~\ref{determining4MfdC} is obtained as a consequence of this result. In Section 6, we give the definition of  $4$-section maps and prove Theorem~\ref{hyperbolic}.

The author would like to thank Masaharu Ishikawa for many discussions and encouragement. He would also like to thank Hironobu Naoe for useful suggestion and  Kenta Hayano for pointing out an error at the preliminary stage of this study. This work was supported in part by the WISE Program for AI Electronics, Tohoku University.

\section{Preliminary}
Let $X$ be a closed orientable smooth $4$--manifold and $f : X \to \mathbf{R}^2$ be a stable map.
Singularities of $f$ are classified into three types: definite folds, indefinite folds and cusps.
The image of indefinite folds is an immersed curve on $\mathbf{R}^2.$ Let $x$ be an indefinite fold and choose a short arc on $\mathbf{R}^2$ that intersects the image of indefinite folds transversely at $f(x)$. The fiber changes along this short arc as shown on the left in Figure~\ref{sing_TFS}. The simple closed curve $c$ shrinking to the point $x$ is called a vanishing cycle of $f$ at $x$.
The image of definite folds is also an immersed curve on $\mathbf{R}^2.$ The fiber changes along a transverse short arc as shown on the middle in Figure~\ref{sing_TFS}. %Let $x$ be an indefinite fold and choose an short arc on $\mathbf{R}^2$ that intersects the image of indefinite fold transversely at $f(x)$. 
A cusp appears at the endpoints of folds. The image of the singular set near the cusp is a cusped curve as shown on the right in the figure. In the case of trisection maps, the folds adjacent to a cusp are always indefinite. Choose a point $p$ near the image of the cusp and draw a transverse short arc from $p$ to each of the two arcs as in the figure. The vanishing cycles $a$ and $b$ for these transverse short arcs intersect transversely at one point on the fiber over $p$.

\begin{figure}[htbp]
\begin{center}
\includegraphics[clip, width=15cm, bb=130 591 561 713]{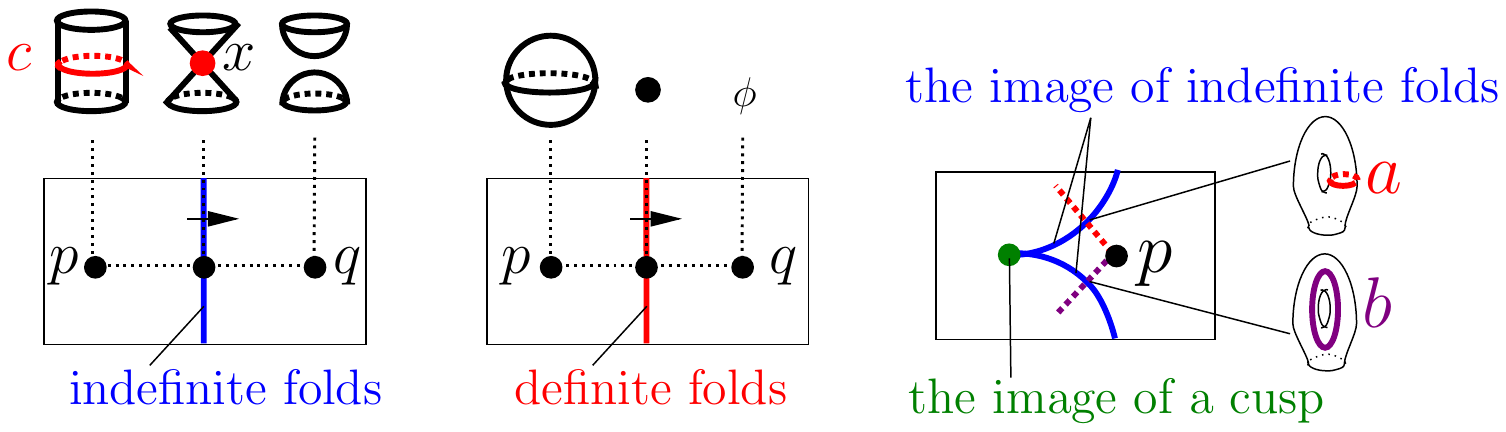}
\end{center}
\caption{Deformation of fibers near singularities.}
\label{sing_TFS}
\end{figure}
%%%%%%%%%%%%%%%%
\begin{comment}
A decomposition $X=X_1\cup X_2\cup X_3$ is called a \emph{$(g,k)$--trisection} if

\begin{itemize}

\item 
for each $i=1,2,3$, there is a diffeomorphism $\phi_i:X_i \to \natural^k(S^1\times D^3)$, and

\item 
for each $i=1,2,3$, taking indices mod $3$, $\phi_i(X_i\cap X_{i+1})=Y_{k,g}^-$ and $\phi_i(X_i\cap X_{i-1})=Y_{k,g}^+$, where $\partial \left(\natural^k (S^1\times D^3)\right) = \sharp^k(S^1\times S^2) = Y_{k,g}^+\cup Y_{k,g}^-$ is a Heegaard splitting of $\sharp^k(S^1\times S^2)$ obtained by stabilizing the standard genus--$k$ Heegaard splitting $g-k$ times.

\end{itemize}

From trisections, we obtain stable mappings $f : X \to \mathbf{R}^2$. Simplified trisection map is a special case of this mapping $f$ as \ref{}.
\end{comment}
%%%%%%%%%%%%%%

\begin{dfn} 
A stable map $f : X \to \mathbf{R}^2$ is called a simplified $(g, k)$-trisection map if the following conditions hold:
\begin{itemize}
\item The singular value set of definite folds is a circle, bounding a disk $D$.
\item The singular value set of indefinite folds consists of $g$ concentric circles on $D$. Each of the inner $g-k$ circles has three outward cusps.
\item The preimage of the point at the center is a closed orientable surface of genus $g$.
\end{itemize}
See Figure~\ref{sim_tri_2}.
\end{dfn}

\begin{figure}[htbp]
\begin{center}
\includegraphics[clip, width=5.0cm, bb=274 499 437 662]{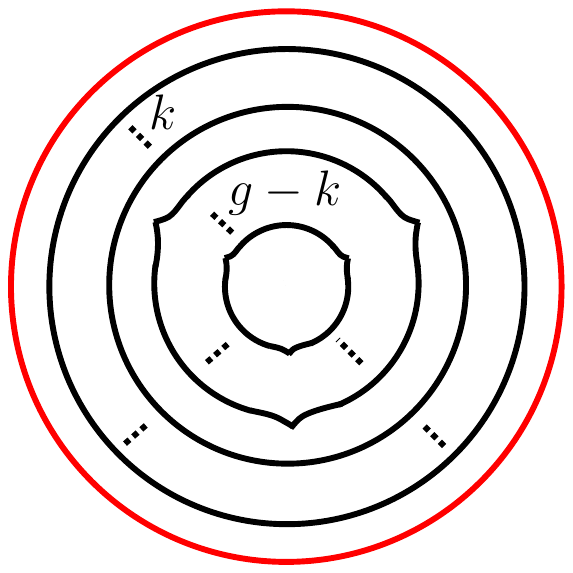}
\end{center}
\caption{A simplified $(g, k)$-trisection map.}
\label{sim_tri_2}
\end{figure}

In this paper, we study simplified $(2, 0)$-trisection maps. %The critical value set is as shown in Figure~\ref{simGenus2TriMonodromyAndPath}. 
Let $e_a, e_b, e_c$ be the edges of the outer, cusped circle as shown on the left in Figure~\ref{simGenus2TriMonodomyAndPath}. The dotted curves in the figure are reference paths. Let $p_1$ be the branch point of the reference paths shown in the figure and let $\Sigma_1$ denote the fiber over $p_1$.%Let $f : X \to \mathbf{R}^2$ be a simplified $(2, 0)$- trisection. The set of critical values is as shown in Figure\ref{}, where the outer and inner most circle of the set of singular values consist of indefinite fold singularities and three cusps.
\begin{figure}
\begin{center}
\includegraphics[clip, width=10cm, bb=128 568 469 713]{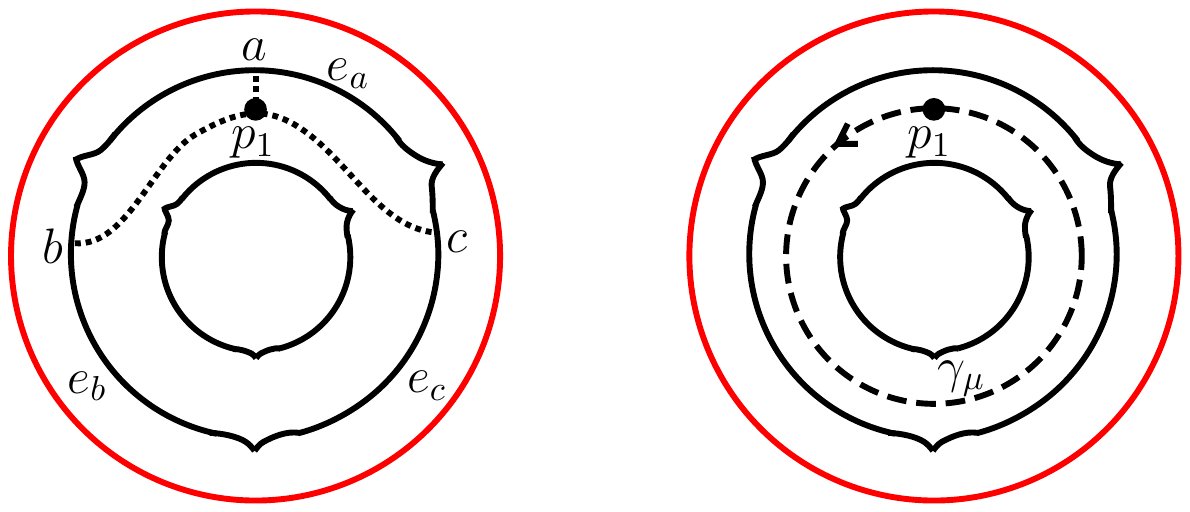}
\end{center}
\caption{Reference paths and the monodromy $\mu$.}
\label{simGenus2TriMonodomyAndPath}
\end{figure}
%There is a simple closed curve on $\Sigma_1$ that shrinks to a point at the intersection point of the reference path and the edge $e_a$.
%If the path intersects the edge $e_a$, there exists a vanishing cycle of indefinitie fold. 
%We denote this vanishing cycle on $\Sigma_1$ by $a$. The vanishing cycles $b$ and $c$ coresponding to the edges $e_b$ and $e_c$, respectively, are defined similarly.
Let $\gamma_\mu$ be a counter-clockwisely oriented circle on $\mathbf{R}^2$ lying between two cusped circles and passing through $p_1$, see the right figure in Figure~\ref{simGenus2TriMonodomyAndPath}. We denote the monodromy diffeomorphism from $\Sigma_1$ to itself along $\gamma_\mu$ by $\mu$. 
By \cite[Lemma 3.6]{hayano}, the monodromy $\mu$ is divided into the following three cases:
\begin{itemize}
\item[(1)] $\mu = \mathrm{id}_{\Sigma_1}$
\item[(2)] $\mu = t^{\pm 1}_d$
\item[(3)] $\mu = t^{\pm 4}_d$,
\end{itemize}
where $d$ is a simple closed curve on $\Sigma_1$ and $t_d$ is the right-handed Dehn twist along $d$. 

%It is known in \cite{Hayano} that $f$ can be deformed by homotopy so that the image of indefinite fold singlarities becomes one of the following cases\\
If $\mu$ is not the identity map, then we can divide the discussion into the following two cases depending on the mutual positions of vanishing cycles:

\begin{itemize}

\item[(A)] $d$ is not parallel to any of $a, b$ and $c$. In this case, it is known in \cite{hayano} that the map can be deformed so that the image of indefinite folds consists of a simple closed curve with four cusps as shown on the left in Figure~\ref{deformation}.
%the image of indefinite folds singlarities consist of a simple closed curves with four cusps as shown in Figure \ref{}.

\item[(B)] $d$ is parallel to one of $a, b$ and $c$. In this case, it is also known in \cite{hayano} that the map can be deformed so that the image of indefinite folds consists of two circles with 3 cusps as shown on the right in Figure \ref{deformation}.

\end{itemize}

\begin{figure}[htbp]
\begin{center}
\includegraphics[clip, width=10cm, bb=128 568 466 713]{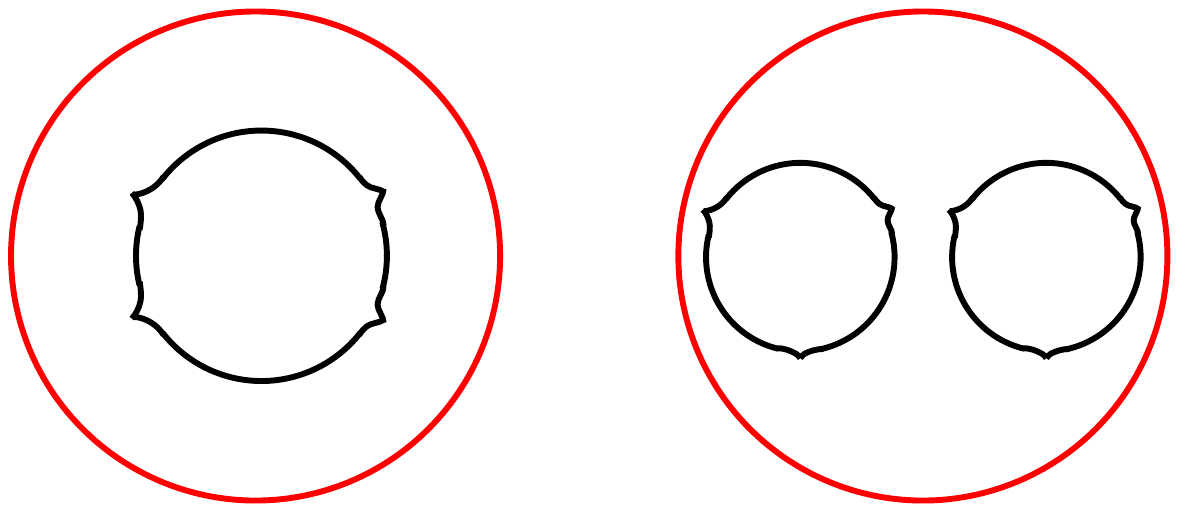}
\end{center}
\caption{Deformation of simplified trisection maps.}
\label{deformation}
\end{figure}
\begin{lem}[See the proof of Theorem 3.9 in \cite{hayano}]\label{deformationLem}
Suppose that $\mu$ is not the identity map. Then the following hold:
\begin{itemize}
\item If $d$ is in case (A), then one of $a, b, c$ intersects $d$ once transversely. 
\item If $\mu = t^{\pm 4}_d$, then $d$ is in case (B). %and $d$ is parallel to one of $a, b, c$. %In particular, $\mu = t^{\pm 1}_d
\end{itemize}
\end{lem}

\section{Classification of vertical 6-tuples}

Let $\gamma_{aa}$ be a properly embedded arc on $f(X) \simeq D^2$ that intersects the image of indefinite folds only at two points on the edge $e_a$ and separates $e_b \cup e_c$ and the inner cusped circle, where $e_a, e_b$ and $e_c$ are arcs in Figure~\ref{simGenus2TriMonodomyAndPath}. The arcs $\gamma_{bb}, \gamma_{cc}$ are defined similarly. See the left figure in Figure~\ref{verTripleArc}.  Let $\gamma_{ba}$ be a properly embedded arc into $f(X) \simeq D^2$ that intersects the image of indefinite folds twice,  at a point on $e_a$ and a point on $e_b$,  and separates $e_c$ and the inner cusped circle. The arcs $\gamma_{cb}, \gamma_{ac}$ are defined similarly. See the right figure in Figure~\ref{verTripleArc}. We set counter-clockwise orientations to these arcs. Set $V_{ij} = f^{-1}(\gamma_{ij})$ for $(i, j) \in \{ (a, a), (b, b), (c, c), (b, a), (c, b), (a, c) \}$ and set orientations to $V_{ij}$ so that it coincides with the product of the orientations of the fiber and the arc $\gamma_{ij}$.
The orientation reversing diffeomorphism of $D^2$ sends 
the $6$-tuple $\begin{pmatrix}
V_{aa} & V_{bb} & V_{cc}\\
V_{ab} & V_{bc} & V_{ca}
\end{pmatrix}$ to  $\begin{pmatrix}
\bar{V}_{aa} & \bar{V}_{cc} & \bar{V}_{bb}\\
\bar{V}_{ac} & \bar{V}_{cb} & \bar{V}_{ba}
\end{pmatrix}$, where $\bar{V}_{ij}$ is the mirror image of $V_{ij}$. This operation corresponds to the exchange of the labels $b$ and $c$. We call it a reflection. In this section, we give a classification of the 6-tuple up to reflection.

In Section 3.1, we study the case where $\mu$ is not the identity. The case where $\mu$ is the identity is an obvious case, which will be explained in Section 3.2.
%In case (B), We assume the label $\{a\}$ to parallel d on the $\Sigma_1$.
\subsection{Case : $\mu \neq \mathrm{id}_{\Sigma_1}$.}
%In case (A),  the curve $d$ intersects one of the vanishing cycles $a, b, c$ once transversely and we set the labels $a, b, c$ so that $a$ is that cycle. 

In this subsection, we assume that $\mu$ is not the identity. %We first study Case (A), where $d$ is not parallel to all of $a, b$ and $c$. 
In Case (A), one of $a, b$ and $c$ intersects $d$ once transversely by Lemma~\ref{deformationLem}.

\begin{thm}\label{thmA}
Suppose $f$ is in Case (A) and the labels $a$, $b$ and $c$ are chosen so that $a$ and $d$ intersect once transversely.
Then the $6$-tuple $\begin{pmatrix}
V_{aa} & V_{bb} & V_{cc}\\
V_{ba} & V_{cb} & V_{ac}
\end{pmatrix}$ is one of the following up to reflection:
\[
\begin{pmatrix}
S^3 & S^3 & L((q - 1)^2, q - 1 + \epsilon)\\
S^1 \times S^2 & L(q - 2, \epsilon) & L(q, -\epsilon)
\end{pmatrix}
, \;
\begin{pmatrix}
S^3 & L(9, 2\epsilon) & L(4, \epsilon)\\
L(2, 1) & L(5, \epsilon) & S^3
\end{pmatrix},
\] where $q \neq 1$ and $\epsilon \in \{-1, 1\}$.
\end{thm}

%Remark that $\epsilon = 1$ if $\mu =t_d$ and $\epsilon = -1$ if $\mu =t^{-1}_d$.

In Case (B),  the curve $d$ is parallel to one of the vanishing cycles $a$, $b$ and $c$.% and we set the labels $a, b, c$ so that $a$ is that cycle. 
\begin{thm}\label{thmB}
Suppose $f$ is in Case (B) and the labels $a$, $b$ and $c$ are chosen so that $a$ is parallel to $d$.
Then the $6$-tuple $\begin{pmatrix}
V_{aa} & V_{bb} & V_{cc}\\
V_{ba} & V_{cb} & V_{ac}
\end{pmatrix}$ is one of the following up to reflection:
\[
%\begin{pmatrix}
%S^1 \times S^2 & S^3 & S^3\\
%S^3 & S^1 \times S^2 & S^3
%\end{pmatrix}, 
\begin{pmatrix}
S^1 \times S^2 & S^3 & S^3\\
S^3 & L(1 + \epsilon, 1) & S^3
\end{pmatrix}
, \;
\begin{pmatrix}
S^1 \times S^2 & L(4, 1) & L(4, 1)\\
S^3 & L(4 + \epsilon, 1) & S^3
\end{pmatrix},
\] where $\epsilon \in \{-1, 1\}$.
\end{thm}

%Let $a, b, c$ be the vanishing cycles on $\Sigma_1$ corresponding to the edges $e_a, e_b, e_c$ along the reference paths in Figure~\ref{simGenus2Tri}, respectively.
Before proving these theorems, we introduce a lemma about Dehn twists. Let $T^2$ be a torus and $\gamma_1 \cdot \gamma_2$ denote the algebraic intersection number of simple closed curves $\gamma_1$ and $\gamma_2$ on an oriented surface.
\begin{lem}\label{DehnTwist}
Let $\gamma \subset T^2$ be an essential simple closed curve, representing an element $\begin{pmatrix}
p\\
q
\end{pmatrix}$ in $H_1(T^2; \mathbf{Z})$. Then, the representation matrix of the right-handed Dehn twist  $t_\gamma$ along $\gamma$ is 
$\begin{pmatrix}
1 - pq & p^2\\
-q^2		& 1 + pq
\end{pmatrix}$.
\end{lem}
\begin{proof}
The map $t_\gamma$ sends $\begin{pmatrix}
1\\
0
\end{pmatrix}$
to
$
 \begin{pmatrix}
1\\
0
\end{pmatrix}
- 
\left(\begin{pmatrix}
1\\
0
\end{pmatrix}
\cdot \begin{pmatrix}
p\\
q
\end{pmatrix}\right)
\begin{pmatrix}
p\\
q
\end{pmatrix} 
=
\begin{pmatrix}
1 - pq \\
-q^2
\end{pmatrix}$
and
$\begin{pmatrix}
0\\
1
\end{pmatrix}$
to
$\begin{pmatrix}
0\\
1\end{pmatrix} -
\left(\begin{pmatrix}
0\\
1\end{pmatrix}
\cdot
\begin{pmatrix}
p\\
q
\end{pmatrix} \right)
\begin{pmatrix}
p\\
q
\end{pmatrix}
=
\begin{pmatrix}
p^2 \\
1 + pq
\end{pmatrix}.$
\end{proof}
\vspace{5mm}
In the following proofs, the notation of double sign correspondence is used.
The notation $M(T, \alpha, \beta)$ means the closed $3$-manifold obtained from a thickening $T \times [0, 1]$ of a torus $T$ by attaching  $2$-handles along essential simple closed curves $\alpha$ on $T \times \{0\}$ and $\beta$ on $T \times \{1\}$ and filling the boundary $3$-balls.
\vspace{5mm}
\begin{proof}[Proof of Theorem~\ref{thmA}]
%Set $a = a_2$. Let $\Sigma_{a_1}$ be the fiber over $p_1$, which is obtained from $\Sigma$ by applying the surgery along the vanishing cycle $a_1$. 
Draw the dotted arcs on $f(X)$ as in Figure~\ref{refPathPrThmA}. The fiber over $p_1$ is a torus, denoted by $\Sigma_1$, and if we move from $p_1$ to a point on the upper edge along the dotted curve, then a simple closed curve $a_2$ on $\Sigma_1$ shrinks to a point. The fiber $\Sigma_1$ over $p_1$ is called a reference fiber, the dotted curve is called a reference path and the simple closed curve $a_2$ is called a vanishing cycle along this reference path. Let $b_2, c_2, a'_2, b'_2$ and $c'_2$ be vanishing cycles on $\Sigma_1$ along the corresponding reference paths shown in Figure~\ref{refPathPrThmA}. In Case (A), $\mu = t^{\pm1}_d$ by Lemma~\ref{deformationLem}.  Since $a_2$ and $b_2$ vanish at the same cusp, choosing the orientations of $a_2$ and $b_2$ suitably, we may identify $H_1(\Sigma_{1};\mathbf{Z})$ with $\mathbf{Z}^2$ so that $a_2$ and $b_2$ represent the elements 
\[
[a_2]=\begin{pmatrix} 1 \\ 0 \end{pmatrix},\quad
[b_2]=\begin{pmatrix} 0 \\ 1 \end{pmatrix}.\quad
\]
Since $b_2$ and $c'_2$ vanish at the same cusp, choosing the orientation of $c'_2$ suitably, we may set $[c'_2]=\begin{pmatrix} -1 \\ q \end{pmatrix}$ with $q \in \mathbf{Z}$.
%Set $a'_2 = \begin{pmatrix} p \\ \epsilon_1 \end{pmatrix}$ with $p \in \mathbf{Z}$ and $\epsilon_1 \in \{\pm1\}$.
Since $d$ intersects $a_2$ once transversely, choosing the orientation of $d$ suitably, we may set $[d]=\begin{pmatrix} r \\ 1 \end{pmatrix}$ with $r \in \mathbf{Z}$. Then, by Lemma~\ref{DehnTwist}, we have $[a_2']=[t_d^{\mp 1}(a_2)]= \begin{pmatrix} p \\ \pm1 \end{pmatrix}$, where $p = 1 \pm r$.
\begin{figure}[htbp]
\begin{center}
\includegraphics[clip, width=6cm, bb=224 567 369 712]{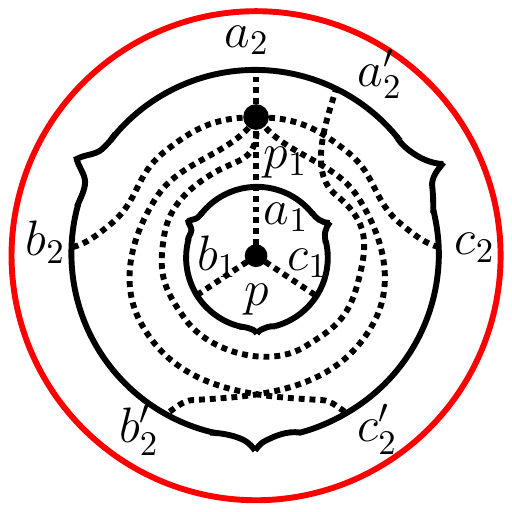}
\end{center}
\caption{}
\label{refPathPrThmA}
\end{figure}
%We choose orientations of the curves as $a_2 \cdot b_2 = b_2 \cdot c'_2 = a_2 \cdot d = 1$. 
%We may identify $H_1(\Sigma_{a_1};\mathbf{Z})$ with $\mathbf{Z}^2$ so that $a_2, b_2, c'_2, a'_2, d$ represent the elements 
%\[
   %[a_2]=\begin{pmatrix} 1 \\ 0 \end{pmatrix},\quad
   %[b_2]=\begin{pmatrix} 0 \\ 1 \end{pmatrix},\quad
   %[c'_2]=\begin{pmatrix} -1 \\ q \end{pmatrix},\quad
   %[a'_2]=\begin{pmatrix} p \\ \epsilon_1 \end{pmatrix},\quad
   %[d]=\begin{pmatrix} r \\ 1 \end{pmatrix},
%\]
%where, $p, q, r \in \mathbf{Z}, \epsilon_1 \in \{ \pm 1 \}$.
%Let $M(\Sigma_1, a_2, c'_2)$ be a closed $3$--manifold obtained by $2$--handles attaching to $\Sigma_1 \times I$ along $a_2 \times \{0\}, c'_2 \times \{1\}$ and capping the boundary spheres by $3$--handles. 
%%
\begin{comment}
$\begin{pmatrix}
1 \\ 0
\end{pmatrix}$,  
$\begin{pmatrix}
0 \\ 1
\end{pmatrix}$. 
$\begin{pmatrix}
-1 \\ q
\end{pmatrix}$, 
$\begin{pmatrix}
p \\ \epsilon_1
\end{pmatrix}$,
$\begin{pmatrix}
r \\ 1
\end{pmatrix}$, respectively.
\end{comment}
%%
\noindent
Since $a'_2$ and $c'_2$ vanish at the same cusp, we may set $a_2'\cdot c_2'=\epsilon_1 \in \{-1, 1\}$. Then we have $pq=\epsilon_1 \mp 1$ as
\[
   \epsilon_1=a_2'\cdot c_2'=\mathrm{det}
\begin{pmatrix}
p & -1\\
\pm1 & q\\
\end{pmatrix}=pq \pm 1.
\]
%%
\begin{comment}
the following holds:
\begin{equation}
1 = |a'_2 \cdot c'_2| = |\mathrm{det}
\begin{pmatrix}
p & -1\\
\epsilon_1 & q\\
\end{pmatrix} |
= |pq + \epsilon_1| .
\end{equation}
Therefore, 
\begin{equation}
pq = \left \{
\begin{array}{l}
-\epsilon_1 + 1\ (a'_2 \cdot c'_2 = 1) = \left \{
\begin{array}{l}
0\ (\epsilon_1 = 1)\\
2\ (\epsilon_1 = -1)
\end{array}
\right.
\\
-\epsilon_1 - 1\ (a'_2 \cdot c'_2 = -1) = \left \{
\begin{array}{l}
-2\ (\epsilon_1 = 1)\\
0\ (\epsilon_1 = -1).
\end{array}
\right.
\end{array}
\right.
\end{equation}
\end{comment}
%%At first, $V_{aa}=M(\Sigma_1, a_2, a'_2)$ is $S^3$ by the assumption of Theorem 1.
%%%%%%%%%%%%%%%%%%%%%%
%%%%%%%%%%%%%%%%%%%%%%
\begin{comment}
The representation matrix of $t_d$ is
$\begin{pmatrix}
1-r & r^2\\
-1 & 1 + r
\end{pmatrix}$ by Lemma \ref{DehnTwist} and $d = \begin{pmatrix}
r\\
1
\end{pmatrix}$.
Since $a_2 = t^{\pm1}_d(a'_2)$, we have \begin{equation}
\left \{
\begin{array}{l}
1 = p(1 \mp r) + r^2 \\
0 = \mp p \pm 1 + r %\ (\mu = t^{\pm}_d, \mathrm{respectively}.)
\end{array}
\right. ,
\end{equation}
which are equivalent to the following one equation:
\begin{equation}\label{sinkCondi}
p = \pm r +1.
\end{equation} %(where $\epsilon = \pm 1$), 
\end{comment}
%%%%%%%%%%%%%%%%%%%%%%%
%%%%%%%%%%%%%%%%%%%%%%%

First, we study the case where either $V_{ba}$ or $V_{ac}$ is $S^1 \times S^2$. By applying a reflection, which means by exchanging the labels $b$ and $c$, if necessary, we assume $V_{ba}$ is $S^1 \times S^2$. 
Since $V_{ba}=M(\Sigma_1, b_2, a'_2)$, $a'_2$ and $b_2$ are parallel, meaning $p = 0$. Hence $\epsilon_1 = a_2'\cdot c_2'= \pm1$ and $r = \mp 1$.

Suppose $\mu=t_d$. Then  $r=-1$. By Lemma~\ref{DehnTwist}, we have
\begin{equation*}
[b'_2] =  [t_d(b_2)] = 
\begin{pmatrix}
2 & 1\\
-1 & 0
\end{pmatrix}
\begin{pmatrix}
0 \\ 1
\end{pmatrix}
= \begin{pmatrix}
1 \\ 0
\end{pmatrix},
\end{equation*} 
\begin{equation*}
[c_2] =  [t_d(c'_2)] = 
\begin{pmatrix}
2 & 1\\
-1 & 0
\end{pmatrix}
\begin{pmatrix}
-1 \\ q
\end{pmatrix}
= \begin{pmatrix}
q-2 \\ 1
\end{pmatrix}.
\end{equation*}
Since it is in Case (A) and $d =\begin{pmatrix}
-1 \\ 1
\end{pmatrix}$, we have $q \neq 1$. 
We obtain $V_{bb} = S^3, V_{ba} = S^1 \times S^2, V_{cb} = L(q -2, 1), V_{ac} = L(q, - 1).$
Remark that the orientation of the arc $\gamma_{ij}$ is counter-clockwise and the orientation of $V_{ij}$ is  given so that it coincides with the product orientation of the fiber and $\gamma_{ij}$.
For $V_{cc} = M(\Sigma, c_2, c'_2)$, setting $T = \begin{pmatrix} 
1 & 0\\
q & 1
\end{pmatrix}$ we have
%. We translate for the pair $(c_2, c'_2)$ to 
$T [c_2] = \begin{pmatrix}q-2\\(q - 1)^2\end{pmatrix}$ and
$T [c'_2] = \begin{pmatrix}- 1\\0 \end{pmatrix}$.
%%%%%%%%%%%%
\begin{comment}
(T c_2, \begin{pmatrix}
- 1\\
0
\end{pmatrix})
= 
(\begin{pmatrix}
1 & 0\\
q & 1
\end{pmatrix}
\begin{pmatrix}
q-2\\
1
\end{pmatrix},
\begin{pmatrix}
- 1\\
0
\end{pmatrix})

 (
\begin{pmatrix}
q-2\\
(q - 1)^2
\end{pmatrix}, 
\begin{pmatrix}
- 1\\
0
\end{pmatrix})
\end{comment}
%%%%%%%%%%%%%%%%
Hence $V_{cc} = \overline{L((q - 1)^2, q - 2)} = L((q - 1)^2, q)$. Obviously, $V_{aa}$ is $S^3$. Thus we have the first $6$-tuple with $\epsilon = 1$ in the assertion. 

Suppose $\mu=t_d^{-1}$. Then $r=1$. By Lemma~\ref{DehnTwist}, we have
%In the case of $a'_2 \cdot c'_2 = \epsilon_1 = 1$, we have the following from (\ref{sinkCondi}):
\begin{equation*}
[b'_2] =  [t^{-1}_d(b_2)] = 
\begin{pmatrix}
2 & -1\\
1 & 0
\end{pmatrix}
\begin{pmatrix}
0 \\ 1
\end{pmatrix}
= \begin{pmatrix}
-1 \\ 0
\end{pmatrix},
\end{equation*}
\begin{equation*}
[c_2] =  [t^{-1}_d(c'_2)] = 
\begin{pmatrix}
2 & -1\\
1 & 0
\end{pmatrix}
\begin{pmatrix}
-1 \\ q
\end{pmatrix}
= \begin{pmatrix}
-2-q \\ - 1
\end{pmatrix}.
\end{equation*}
Since it is in Case (A) and $d =\begin{pmatrix}
1 \\ 1
\end{pmatrix}$, we have $q \neq -1$.
We obtain $V_{bb} = S^3, V_{ba} = S^1 \times S^2$, $V_{cb} = L(q + 2, 1)$ and $V_{ac} = L(q,  - 1)$.
For $V_{cc} = M(\Sigma, c_2, c'_2)$, setting $T = \begin{pmatrix} 
0 & -1\\
1 & -2 - q
\end{pmatrix}$ we have $T [c_2] =  \begin{pmatrix}1\\0\end{pmatrix}$ and $T [c'_2] = \begin{pmatrix}- q\\- (q + 1)^2\end{pmatrix}$.
Hence $V_{cc} = L((q + 1)^2, q)$. By replacing $q$ into $-q$, we have the 6-tuple 
\[
\begin{pmatrix}
V_{aa} & V_{bb} & V_{cc}\\
V_{ba} & V_{cb} & V_{ac}
\end{pmatrix}
=
\begin{pmatrix}
S^3 & S^3 & L((q - 1)^2, - q)\\
S^1 \times S^2 & L(q - 2, -1) & L(q,  1)
\end{pmatrix}
\] with $q \neq 1$.
This is the first 6-tuple with $\epsilon = -1$ in the assertion.
%So we have the following equation from (1) and (2). 

Next, we consider the case where both of  $V_{ba}$ and $V_{ac}$ are not $S^1 \times S^2$. Since $M(\Sigma_1, b_2, a'_2)$ and $M(\Sigma_1,a_2, c'_2)$ are not $S^1 \times S^2$, $pq \neq 0$. Hence $pq = \epsilon_1 \mp 1$ is either $-2$ or $2$. In particular, $p\in \{-2, -1, 1, 2\}$.% which means $(\epsilon_1, \epsilon_2) = (-1, 1), (1, -1)$.By equalities (\ref{sinkCondi}), we have $p = \pm 1, \pm2$.%$1 = p - \epsilon_1r$.
%we can assume $|p| = 1$ by reflection of $\mathbf{R}^2$.

If $p = 1$, then $r = 0$ as $p = 1 \pm r$. We have $q = \mp2$ since $q = \epsilon \mp 1$ and $q \neq 0$, and $[d] = 
\begin{pmatrix}
0\\
1
\end{pmatrix}$. 
Since $b_2$ is parallel to $d$, it is not in Case (A)
%%%%%%%%%%%%%%%%%%%%%%%%%%%%%%
%%%%%%%%%%%%%%%%%%%%%%%%%%%%%%
\begin{comment}
By Lemma~\ref{deformationLem}, the representation matrix of $\mu = t^{\pm 1}_d$ is
$\begin{pmatrix}
1 & 0\\
\mp 1 & 1
\end{pmatrix}$. Therefore, 
\[[b'_2] = [\mu(b_2)] = 
\begin{pmatrix}
1 & 0\\
\mp 1 & 1
\end{pmatrix}
\begin{pmatrix}
0\\
1
\end{pmatrix}
=
\begin{pmatrix}
0\\
1
\end{pmatrix},  
\]
\[
[c_2] = [\mu(c'_2)] =
\begin{pmatrix}
1 & 0\\
\mp 1 & 1
\end{pmatrix}
\begin{pmatrix}
-1\\
\mp 2
\end{pmatrix} =
\begin{pmatrix}
-1\\
\mp 1
\end{pmatrix}.
\]
Thus, we obtain the $V_{aa} = S^3, V_{bb} = S^1 \times S^2, V_{cc} = S^3, V_{ba} = S^3, V_{cb} = S^3$ and $V_{ac} = L(2, 1)$. This is the reflection of the second $6$-tuple in the assertion.
\end{comment}
%%%%%%%%%%%%%%%%%%%%%%%%%%
%%%%%%%%%%%%%%%%%%%%%%%%%%

If $p = -1$, we have $r = \mp 2$. Since $q = \pm 1 - \epsilon_1$ and $q \neq 0$, we have $q =\pm 2$. Since $[d] = 
\begin{pmatrix}
\mp 2\\
1
\end{pmatrix}$, the representation matrix of $\mu = t^{\pm1}_d$ is
$\begin{pmatrix}
3 & \pm4\\
\mp 1 & -1
\end{pmatrix}$ by Lemma~\ref{DehnTwist}. Hence we have
\[
[b'_2] = [\mu(b_2)] = 
\begin{pmatrix}
3 & \pm4\\
\mp 1 & -1
\end{pmatrix}
\begin{pmatrix}
0\\
1
\end{pmatrix}
=
\begin{pmatrix}
\pm4\\
-1
\end{pmatrix},
\] 
\[
[c_2] = [\mu(c'_2)] =
\begin{pmatrix}
3 & \pm4\\
\mp 1 & -1
\end{pmatrix}
\begin{pmatrix}
-1\\
\pm 2
\end{pmatrix} =
\begin{pmatrix}
5\\
\mp 1
\end{pmatrix}.
\] Setting $T = \begin{pmatrix} 
0 & \mp 1\\
\pm 1 & 5
\end{pmatrix}$, we have $T [c_2] =  \begin{pmatrix}1\\0\end{pmatrix}$ and $T [c'_2] = \begin{pmatrix}- 2\\\pm 9\end{pmatrix}$. Thus we obtain $V_{aa} = S^3, V_{bb} = L(\pm4, -1), V_{cc} = L(\pm9, -2), V_{ba} = S^3, V_{cb} = L(\pm5, -1)$ and $V_{ac} = L(2, 1)$. 
This is the reflection of the second $6$-tuple in the assertion.

If $p = 2$, then we have $r = \pm1$. Since $2q = \epsilon_1 \mp 1$ and $q \neq 0$, we have $q = \mp 1$. Since $[d] = \begin{pmatrix}
\pm 1 \\
1
\end{pmatrix},$ the representation matrix of $\mu = t^{\pm 1}_{d}$ is $\begin{pmatrix}
0 & \pm 1 \\
\mp1 & 2
\end{pmatrix}$ by Lemma~\ref{DehnTwist}. Then we have
\[
[b'_2] = [\mu(b_2)] = 
\begin{pmatrix}
0 & \pm 1 \\
\mp1 & 2
\end{pmatrix}
\begin{pmatrix}
0\\
1
\end{pmatrix}
=
\begin{pmatrix}
\pm 1\\
2
\end{pmatrix},
\] 
\[
[c_2] = [\mu(c'_2)] =
\begin{pmatrix}
0 & \pm 1 \\
\mp1 & 2
\end{pmatrix}
\begin{pmatrix}
-1\\
\mp 1
\end{pmatrix} =
\begin{pmatrix}
-1\\
\mp 1
\end{pmatrix}.
\]
Since $c_2$ is parallel to $d$, it is not in Case (A).

%%%%%%%%%%%%%%%%%%%%%%%%%%%
%%%%%%%%%%%%%%%%%%%%%%%%%%%
\begin{comment}

Thus we obtain $V_{aa} = S^3, V_{bb} = S^3, V_{cc} = S^1 \times S^2, V_{ba} = L(2, 1), V_{cb} = S^3$ and $V_{ac} = S^3$. This is the second $6$-tuple in the assertion.

\end{comment}
%%%%%%%%%%%%%%%%%%%%%%%%%%%
%%%%%%%%%%%%%%%%%%%%%%%%%%%
If $p = -2$, then we have $r = \mp3$. Since $2q = \pm 1 - \epsilon_1$ and $q \neq 0$ we have $q =\pm 1$. Since $[d] = \begin{pmatrix}
\mp3 \\
1
\end{pmatrix}$, the representation matrix of $\mu = t^{\pm1}_{d}$ is $
\begin{pmatrix}
4 & \pm9 \\
\mp1 & -2
\end{pmatrix}$ by Lemma~\ref{DehnTwist}. Then we have 
\[
[b'_2] = [\mu(b_2)] = 
\begin{pmatrix}
4 & \pm9 \\
\mp1 & -2
\end{pmatrix}
\begin{pmatrix}
0\\
1
\end{pmatrix}
=
\begin{pmatrix}
\pm9\\
-2
\end{pmatrix},
\] 
\[
[c_2] = [\mu(c'_2)] =
\begin{pmatrix}
4 & \pm9 \\
\mp1 & -2
\end{pmatrix}
\begin{pmatrix}
-1\\
\pm 1
\end{pmatrix} =
\begin{pmatrix}
5\\
\mp 1
\end{pmatrix}.
\]
Thus we obtain $V_{aa} = S^3,  V_{bb} = L(\pm9, -2), V_{ba} = L(2, 1)$, and $V_{ac} = S^3$. For $V_{cc}$ and $V_{cb}$, multiplying $T = \begin{pmatrix}
1 & \pm 4\\
\pm 1 & 5
\end{pmatrix},$ we have $T [c_2]  = \begin{pmatrix}
1\\
0
\end{pmatrix}, 
T [c'_2]  = \begin{pmatrix}
3\\
\pm 4
\end{pmatrix}$ and
$T [b_2]  = \begin{pmatrix}
\pm4\\
5
\end{pmatrix}.$ Hence $V_{cc} = L(\pm4, -1), V_{cb} = L(5, \mp1)$. This is the second $6$-tuple in the assertion. This completes the proof.
%%%%%%%%%%%%%%%%%%%%%%%%%%%
%%%%%%%%%%%%%%%%%%%%%%%%
\begin{comment}
If $p = -2, q = 1$, then we have $r = -3$ by \eqref{sinkCondi}. Then, we check the monodromy $\mu$ is $ t_{d}$. So, $d = \begin{pmatrix}
-3 \\
1
\end{pmatrix}, t_{d} =
\begin{pmatrix}
4 & 9 \\
-1 & -2
\end{pmatrix}$, Then we have the following:
\[
b'_2 = \mu(b_2) = 
\begin{pmatrix}
4 & 9 \\
-1 & -2
\end{pmatrix}
\begin{pmatrix}
0\\
1
\end{pmatrix}
=
\begin{pmatrix}
9\\
-2
\end{pmatrix},
\] 
\[
c_2 = \mu(c'_2) =
\begin{pmatrix}
4 & 9 \\
-1 & -2
\end{pmatrix}
\begin{pmatrix}
-1\\
1
\end{pmatrix} =
\begin{pmatrix}
5\\
-1
\end{pmatrix}.
\]
Thus we obtain $V_{aa} = S^3, V_{bb} = L(9, 2), V_{cc} = L(4, 1), V_{ba} = L(2, 1), V_{cb} = L(5, 1)$ and $V_{ac} = S^3$.

If $p = 2, q = -1$, then we have $r = 1$ by \eqref{sinkCondi}. Then, we check the monodromy $\mu$ is $ t_{d}$. So, $d = \begin{pmatrix}
1 \\
1
\end{pmatrix}, t_{d} =
\begin{pmatrix}
0 & 1 \\
-1 & 2
\end{pmatrix}$, Then we have the following:
\[
b'_2 = \mu(b_2) = 
\begin{pmatrix}
0 & 1 \\
-1 & 2
\end{pmatrix}
\begin{pmatrix}
0\\
1
\end{pmatrix}
=
\begin{pmatrix}
1\\
2
\end{pmatrix},
\] 
\[
c_2 = \mu(c'_2) =
\begin{pmatrix}
0 & -1 \\
1 & 2
\end{pmatrix}
\begin{pmatrix}
-1\\
1
\end{pmatrix} =
\begin{pmatrix}
-1\\
1
\end{pmatrix}.
\]
Thus we obtain $V_{aa} = S^3, V_{bb} = S^3, V_{cc} = S^1 \times S^2, V_{ba} = L(2, 1), V_{cb} = S^3$ and $V_{ac} = S^3$.
\end{comment}
%%%%%%%%%%%%%%%%%%%%%%%%%%%%%%
%%%%%%%%%%%%%%%%%%%%%%%%%%%%%%t
\end{proof}

%%%%%%%%%%%%%%%%%%%%%%%%%%%%%%
%%%%%%%%%%%%%%%%%%%%%%%%%%%%%%
\begin{comment}
So, we obtain 
\[
\begin{pmatrix}
V_{aa} & V_{bb} & V_{cc}\\
V_{ab} & V_{bc} & V_{ca}
\end{pmatrix}
=
\begin{pmatrix}
S^3 & L((p - 1)^2, p) & S^3\\
L(p, 1) & L(p - 2, 1) & S^1 \times S^2
\end{pmatrix}.
\]
After all, if $V_{ca}$ is a $S^1 \times S^2$, we conclude 
\[
\begin{pmatrix}
V_{aa} & V_{bb} & V_{cc}\\
V_{ab} & V_{bc} & V_{ca}
\end{pmatrix}
=
\begin{pmatrix}
S^3 & L((p - 1)^2, p) & S^3\\
L(p, 1) & L(p - 2, 1) & S^1 \times S^2
\end{pmatrix}.
\]

\begin{rem}
In the case of $V_{ac} = S^1 \times S^2$, Since $V_{ac}=M(\Sigma_1, a_2, c'_2)$ , $a_2$ and $c'_2$ are parallel, meaning $q = 0$. Then the 6-tuple becomes that
$V_{ba}, V_{ac}$ and $V_{bb}, V_{cc}$ are interchanged each other.
This is same as the 6-tuple in the assertion since $L((p - 1)^2, p) \simeq L((p - 1)^2, p-2)$.
\end{rem}
\end{comment}
%%%%%%%%%%%%%%%%%%%%%%%%%%%%%%
%%%%%%%%%%%%%%%%%%%%%%%%%%%%%%
\begin{proof}[Proof of Theorem~\ref{thmB}]
Let $a_2, b_2, c_2, a'_2, b'_2$ and $c'_2$ be the vanishing cycles on $\Sigma_{1}$ and identify $H_1(\Sigma_{1};\mathbf{Z})$ with $\mathbf{Z}^2$ so that $a_2$ and $b_2$ represent the elements 
\[
[a_2]=\begin{pmatrix} 1 \\ 0 \end{pmatrix},\quad
[b_2]=\begin{pmatrix} 0 \\ 1 \end{pmatrix}\quad
\] as in the proof of Theorem~\ref{thmA}. 
By the assumption, we may set $[d]=\begin{pmatrix} 1 \\ 0 \end{pmatrix}$. Hence $[a'_2] = [t^{\mp1}_{d}(a_2)] = \begin{pmatrix} 1 \\ 0 \end{pmatrix}$. Since $b_2$ and $c
'_2$ vanish at the same cusp and $a'_2$ and $c'_2$ vanish at the same cusp, we may set $[c'_2]=\begin{pmatrix} -1 \\ \epsilon_2 \end{pmatrix}$ with $\epsilon_2 \in \{-1, 1\}$. 

%%%%%%%%%%%%%%%%%%%%%%%%%%%%%%%
%%%%%%%%%%%%%%%%%%%%%%%%%%%%%%%
\begin{comment}
First, we prove the case where $V_{cb}$ is $S^1 \times S^2$. If $\mu = t^{\pm4}_d$, then $c_2 = \mu(c'_2) = t^{\pm4}_d (c'_2) = 
\begin{pmatrix}
1 & \pm4\\
0 & 1
\end{pmatrix}
\begin{pmatrix}
-1\\
\epsilon_1
\end{pmatrix} =
\begin{pmatrix}
-1\pm 4\epsilon_1\\
\epsilon_1
\end{pmatrix}.$ Thus $V_{cb}$ cannot be $S^1 \times S^2$. 
\end{comment}
%%%%%%%%%%%%%%%%%%%%%%%%%%%%%%%%%
%%%%%%%%%%%%%%%%%%%%%%%%%%%%%%%%%
Suppose $\mu = t^{\pm1}_d.$ Then we have 
\begin{equation*}\label{proof2B}
[b'_2] = [\mu(b_2)] = [t^{\pm1}_d (b_2)] = 
\begin{pmatrix}
1 & \pm1\\
0 & 1
\end{pmatrix}
\begin{pmatrix}
0\\
1
\end{pmatrix}
=
\begin{pmatrix}
\pm1\\
1
\end{pmatrix},
\end{equation*}  
\begin{equation*}%\label{proof2C}
[c_2] = [\mu(c'_2)] = [t^{\pm1}_d (c'_2)] = 
\begin{pmatrix}
1 & \pm1\\
0 & 1
\end{pmatrix}
\begin{pmatrix}
-1\\
\epsilon_2
\end{pmatrix} =
\begin{pmatrix}
-1\pm \epsilon_2\\
\epsilon_2
\end{pmatrix}.
\end{equation*} 
Thus we obtain $V_{aa} = S^1 \times S^2, V_{bb} = S^3, V_{cc} = S^3, V_{ba} = S^3, V_{cb} = L(-1\pm \epsilon_2, \epsilon_2)$ and $V_{ac} = S^3$. If $\pm\epsilon_2 =1$ then it is the first $6$-tuple in the assertion with $\epsilon = -1$ and if $\pm \epsilon_2 = -1$ then it is the with $\epsilon = -1$. 

%%%%%%%%%%%%%%%%%%%%%%%%%%%%%%%%%%%%%
%%%%%%%%%%%%%%%%%%%%%%%%%%%%%%%%%%%%
\begin{comment}
Next, we prove the case where $V_{cb}$ is not $S^1 \times S^2$. %In the case of $\mu = t^{\pm1}_d$. 
If the monodromy $\mu = t^{\pm1}_d$, then we have the followings:
\begin{equation}\label{proof2B}
b'_2 = \mu(b_2) = t^{\pm1}_d (b_2) = 
\begin{pmatrix}
1 & \pm1\\
0 & 1
\end{pmatrix}
\begin{pmatrix}
0\\
1
\end{pmatrix}
=
\begin{pmatrix}
\pm1\\
1
\end{pmatrix},
\end{equation} 
\begin{equation}\label{proof2C}
c_2 = \mu(c'_2) = t^{\pm1}_d (c'_2) = 
\begin{pmatrix}
1 & \pm1\\
0 & 1
\end{pmatrix}
\begin{pmatrix}
-1\\
\epsilon_1
\end{pmatrix} =
\begin{pmatrix}
-1\pm \epsilon_1\\
\epsilon_1
\end{pmatrix}.
\end{equation} 
Since $V_{cb}$ is not $S^1 \times S^2$, $-1 \pm \epsilon_1$ cannot be $0$.
So we have $(b'_2, c_2) = \left(
\begin{pmatrix}
\pm 1\\
1
\end{pmatrix}, 
\begin{pmatrix}
-2\\
\mp 1
\end{pmatrix}\right)$.
So we have $(V_{cc}, V_{cb}) = (S^3, L(2, 1))$. 
Thus we obtain
\[
\begin{pmatrix}
V_{aa} & V_{bb} & V_{cc}\\
V_{ba} & V_{cb} & V_{ac}
\end{pmatrix}
=
\begin{pmatrix}
S^1 \times S^2 & S^3 & S^3\\
S^3 & L(2, 1) & S^3
\end{pmatrix}.
\]
\end{comment}
%%%%%%%%%%%%%%%%%%%%%%%%%%%%%%%%
%%%%%%%%%%%%%%%%%%%%%%%%%%%%%%%%

If $\mu = t^{\pm4}_d$, then we have
\begin{equation*}%\label{proof2B_2}
[b'_2] = [\mu(b_2)] = [t^{\pm4}_d (b_2)] = 
\begin{pmatrix}
1 & \pm4\\
0 & 1
\end{pmatrix}
\begin{pmatrix}
0\\
1
\end{pmatrix}
=
\begin{pmatrix}
\pm4\\
1
\end{pmatrix},
\end{equation*}  
\begin{equation*}%\label{proof2C_2}
[c_2] = [\mu(c'_2)] = [t^{\pm4}_d (c'_2)] = 
\begin{pmatrix}
1 & \pm4\\
0 & 1
\end{pmatrix}
\begin{pmatrix}
-1\\
\epsilon_2
\end{pmatrix} =
\begin{pmatrix}
-1\pm 4\epsilon_2\\
\epsilon_2
\end{pmatrix}.
\end{equation*} 
So we have $V_{aa} = S^1 \times S^2, V_{bb} = L(\pm4, 1), V_{ba} = S^3,  V_{ac} = S^3.$ For $V_{cc}$ and $V_{cb}$, multiplying $T =\begin{pmatrix}
-1 & \pm 4\\
-\epsilon_2 & -1 \pm 4\epsilon_2
\end{pmatrix}$, we have 
$T [c_2] = \begin{pmatrix}
1\\
0
\end{pmatrix}, 
T [c'_2] = \begin{pmatrix}
1 \pm 4 \epsilon_2 \\
\pm 4
\end{pmatrix}$ and
$T [b_2] = \begin{pmatrix}
\pm 4\\
-1 \pm 4 \epsilon_2
\end{pmatrix}$.
Hence $V_{cc} = L(\pm4, 1)$ and $V_{cb} = L(-1 \pm 4 \epsilon_2, \pm 4) = L(\pm 4 - \epsilon_2, 1)$. 
%%%%%%%%%%%%%%%%%%%%%%%%%
%%%%%%%%%%%%%%%%%%%%%%%%%
\begin{comment}
\[
(b'_2, c_2) = \left(
\begin{pmatrix}
4\\
1
\end{pmatrix},
\begin{pmatrix}
3\\
1
\end{pmatrix}\right),
\left(\begin{pmatrix}
-4\\
1
\end{pmatrix},
\begin{pmatrix}
-5\\
1
\end{pmatrix}\right),
\left(\begin{pmatrix}
4\\
1
\end{pmatrix},
\begin{pmatrix}
-5\\
-1
\end{pmatrix}\right),
\left(\begin{pmatrix}
-4\\
1
\end{pmatrix},
\begin{pmatrix}
3\\
-1
\end{pmatrix}\right).
\]
\end{comment}
%%%%%%%%%%%%%%%%%%%
%%%%%%%%%%%%%%%%%%%
Thus, we obtain
\[
\begin{pmatrix}
V_{aa} & V_{bb} & V_{cc}\\
V_{ba} & V_{cb} & V_{ac}
\end{pmatrix}
=
\begin{pmatrix}
S^1 \times S^2 & L(4, 1) & L(4, 1)\\
S^3 & L(4 \mp \epsilon_2, 1) & S^3
\end{pmatrix}
\]
after applying the reflection if necessary. This completes the proof.
 \end{proof}
 \subsection{Case : $\mu$ is trivial.}
 \begin{lem}\label{caseId}
If $\mu$ is the identity map, then the 6-tuple of vertical manifolds is
\[
\begin{pmatrix}
V_{aa} & V_{bb} & V_{cc}\\
V_{ba} & V_{cb} & V_{ac}
\end{pmatrix} =
\begin{pmatrix}
S^1 \times S^2 & S^1 \times S^2 & S^1 \times S^2\\
S^3 & S^3 &S^3
\end{pmatrix}
\]
\end{lem}
\begin{proof}
Since $\mu$ is the identity, the vertical manifold over the arc $\omega_1$ shown on the left in Figure~\ref{caseIdPathMove} is homeomorphic to the one over the arc $\omega'_1$ shown on the right. Therefore, $V_{aa}$ is homeomorphic to $S^1 \times S^2$. The other vertical $3$-manifolds can be determined by the same manner.
\end{proof}
\begin{figure}[htbp]
\begin{center}
\includegraphics[clip, width=9cm, bb=128 548 482 712]{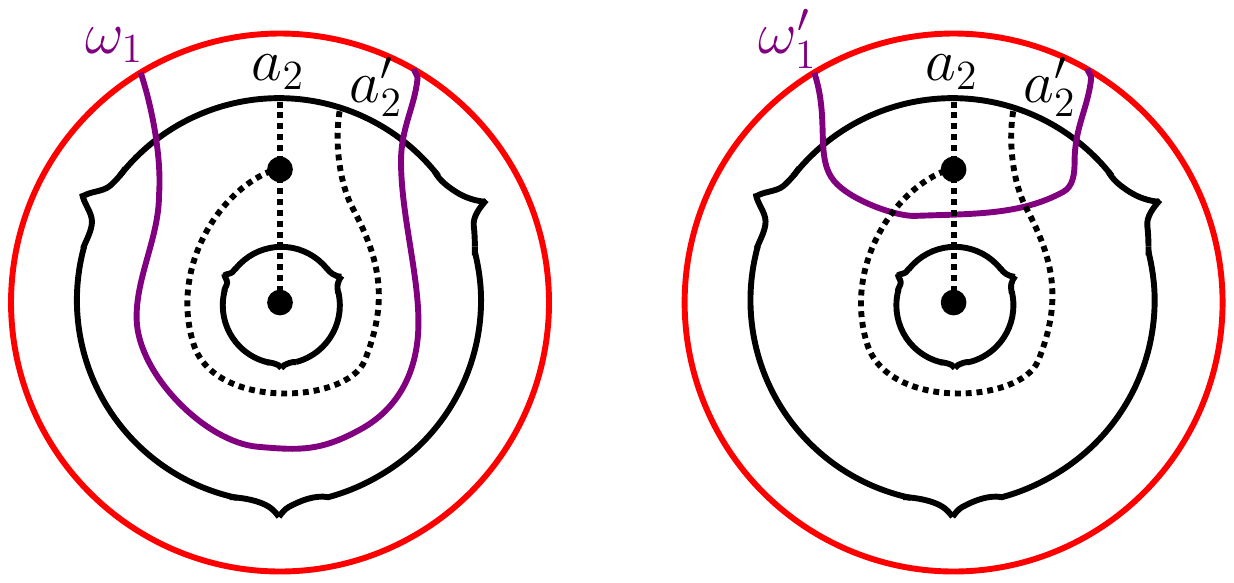}
\caption{Case : $\mu$ is the identity.}
\label{caseIdPathMove}
\end{center}
\end{figure}
 \section{Proof of Theorem~\ref{preimage} }

\begin{lem}\label{moves}
Let $f:X\to\mathbf{R}^2$ be a simplified $(2, 0)$-trisection map and $\omega \subset f(X)$ be a proper, simple, generic arc oriented counter-clockwisely. Let $M_i$ $(i =1, 2)$ be the oriented closed $3$-manifold obtained as the preimage of the arc $\omega_i$ in Figure~\ref{partsOfVerMfd}, where the orientation coincides with the product of those of $\omega_i$ and the fiber surface. Then, $f^{-1}(\omega)$ is either $S^3$ or a connected sum of finite copies of $S^1 \times S^2, M_1$ and $M_2$ and their mirror images.

\begin{figure}[htbp]
\begin{center}
\includegraphics[clip, width=8cm, bb=190 621 405 713]{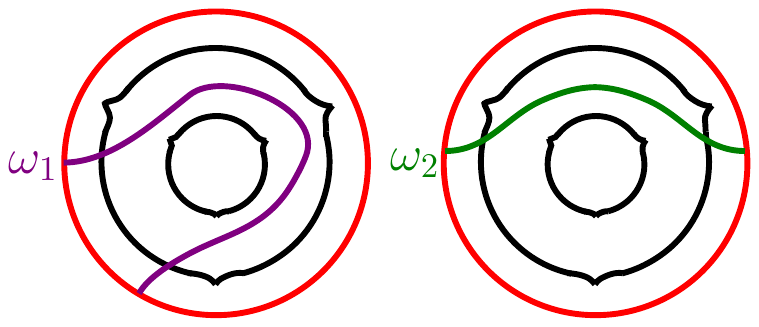}
\caption{The arcs $\omega_1$ and $\omega_2$.}
\label{partsOfVerMfd}
\end{center}
\end{figure}
\end{lem}

\begin{proof}
We first introduce two local moves shown in Figure \ref{move}.
The move on the left is named a {\it bigon move} and the one on the right is a {\it cusp move}.
The vertical manifold $f^{-1}(\gamma_1)$ is homeomorphic to $f^{-1}(\gamma'_1)$, and the vertical manifold $f^{-1}(\gamma_2)$ is  a connected sum of $f^{-1}(\gamma'_2)$ and $S^1 \times S^2$. 

\begin{figure}[htbp]
\begin{center}
\includegraphics[clip, width=10.5cm, bb=130 598 516 713]{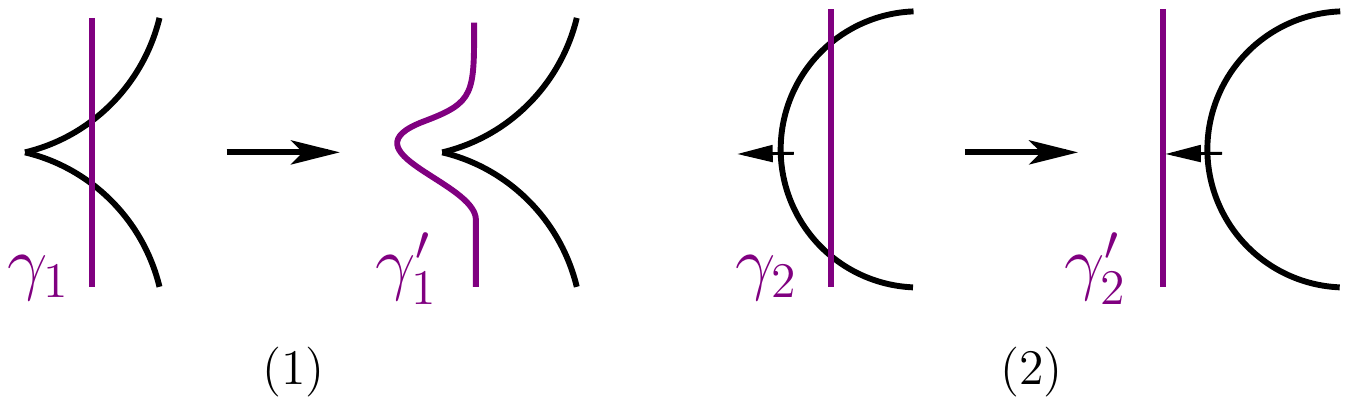}
\caption{(1) a cusp move, (2)  a bigon move.}
\label{move}
\end{center}
\end{figure}

The image of the indefinite folds consists of two concentric circles with cusps. Let $C_1$ be the inner circle and $C_2$ be the outer one. If $\omega$ passes near a cusp of $C_1$ as shown in Figure~\ref{move} (1), then we apply a cusp move to $\omega$. If $\omega$ and $C_1$ bound a disc as shown in Figure~\ref{move} (2), then we apply a bigon move to $\omega$. Applying these operations inductively, we may assume that $\omega$ does not intersect $C_1$.

Next we focus on the arcs on $\omega$ whose endpoints are on $C_2$ and which lies on the annulus between $C_2$ and the image of definite folds. For an outermost arc of these arcs we apply the operation shown in Figure~\ref{move2}.  Suppose that the arc $\omega$ is changed  into the union of the arcs $\omega'$ and $\omega''$ by this operation. Since the preimage of the intersection of an arc and the image of the definite folds is a set of points, the inverse of the above operation induces the following surgery for the $3$-manifolds $f^{-1}(\omega')$ and $f^{-1}(\omega'')$; remove a $3$-ball from each of $f^{-1}(\omega')$ and $f^{-1}(\omega'')$ and glue the two $2$-spheres appearing on the boundaries, which produces a connected sum of $f^{-1}(\omega')$ and $f^{-1}(\omega'')$. We apply this operation inductively so that there are no arcs whose endpoints are on $C_2$ and which lies on the annulus between $C_2$ and the image of definite folds. Then, applying cusp and bigon moves to the arcs decomposed by these operations inductively, we see that the preimages of the resulting arcs are finite copies of $S^1 \times S^2$, $M_1$ and $M_2$ and their mirror images. Thus we have the assertion.%We first consider the arc $\omega_1$ and $C_1$ bounding the disc
\end{proof}
\begin{figure}[htbp]
\begin{center}
\includegraphics[clip, width=10.5cm]{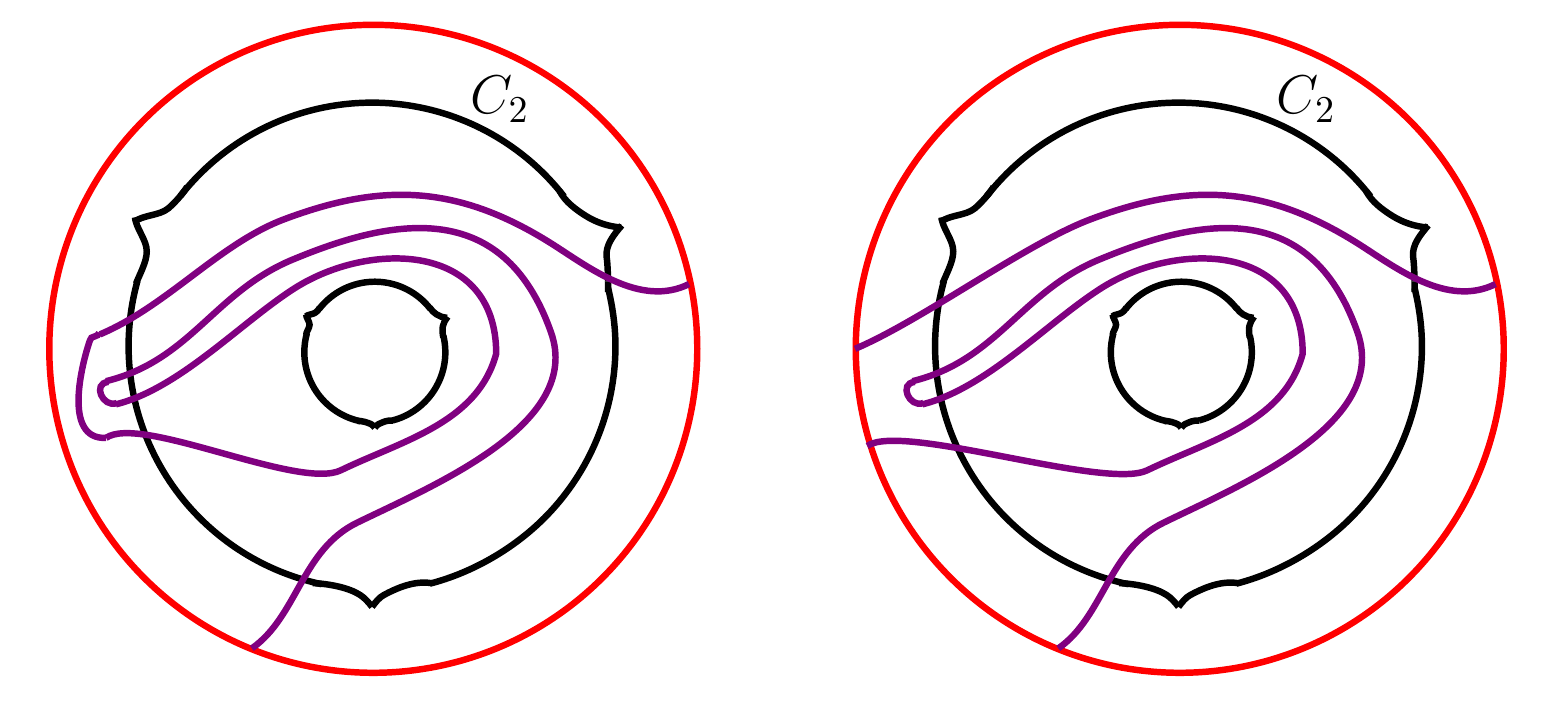}
\caption{Cutting the arc.}
\label{move2}
\end{center}
\end{figure}

\begin{proof}[Proof of Theorem~\ref{preimage}.]
Consider the first $6$-tuple 
\[
\begin{pmatrix}
V_{aa} & V_{bb} & V_{cc}\\
V_{ba} & V_{cb} & V_{ac}
\end{pmatrix} =
\begin{pmatrix}
S^3 & S^3 & L((q - 1)^2, q - 1 + \epsilon)\\
S^1 \times S^2 & L(q - 2, \epsilon) & L(q, -\epsilon)
\end{pmatrix}
\]
in Theorem~\ref{thmA}. By Lemma~\ref{moves}, the vertical manifolds in this case are connected sums of these manifolds and their mirror images. Checking the positions of the arcs $\gamma_{aa}, \gamma_{bb}, \gamma_{cc}, \gamma_{ba}, \gamma_{cb}$ and $\gamma_{ac}$, we see that possible combinations are
\[
(\gamma_{aa}, \gamma_{ba}), (\gamma_{aa}, \gamma_{ac}), (\gamma_{bb}, \gamma_{ba}), (\gamma_{bb}, \gamma_{cb}), (\gamma_{cc}, \gamma_{cb}), (\gamma_{cc}, \gamma_{ac}),
\] cf. Figure~\ref{20200807Mirror}.
The vertical manifolds corresponding to $(\gamma_{aa}, \gamma_{ba})$ are connected sums of finite copies of $S^1 \times S^2$.
The vertical manifolds corresponding to $(\gamma_{aa}, \gamma_{ac})$ are $\#^{\ell + \epsilon_1, \ell}L(q, -\epsilon)$ up to summands $S^1 \times S^2$ and orientation reversal. We may determine the vertical manifolds corresponding to $(\gamma_{bb}, \gamma_{ba})$ and $(\gamma_{bb}, \gamma_{cb})$ by the same manner.

A vertical manifold corresponding to $(\gamma_{cc}, \gamma_{cb})$ is the preimage of the arc in Figure~\ref{20200807Mirror} with closing the endpoints of the arcs suitably, where $r$ and $s$ are the numbers of the arcs.
Checking orientations of the arcs carefully, we conclude that it is $\#^{\ell + \epsilon_1, \ell}L((q-1)^2, q-1 + \epsilon)\#^{m, m+ \epsilon_2} L(q - 2, \epsilon)$ up to summands $S^1 \times S^2$ and orientation reversal, where $2\ell + \epsilon_1 = r$ and $2m + \epsilon_2 = s$. 

For a vertical manifold corresponding to $(\gamma_{cc}, \gamma_{ac})$ we may conclude that it is $\#^{\ell + \epsilon_1, \ell}L((q-1)^2, q-1 + \epsilon)\#^{m, m+ \epsilon_2} L(q , -\epsilon)$ up to summands $S^1 \times S^2$ and orientation reversal.%, where $2\ell + \epsilon_1 = r$ and $2m + \epsilon_2 = s$.

We may determine the vertical $3$-manifolds for the other $6$-tuples in Theorem~\ref{thmA} and also Theorem~\ref{thmB} by the same manner.
%The assertion follows from Theorem~\ref{thmA}, Theorem~\ref{thmB}, Lemma~\ref{caseId} and Lemma~\ref{moves}.
\end{proof}

\begin{figure}[htbp]
\begin{center}
\includegraphics[clip, width=5.5cm, bb=181 482 412 713]{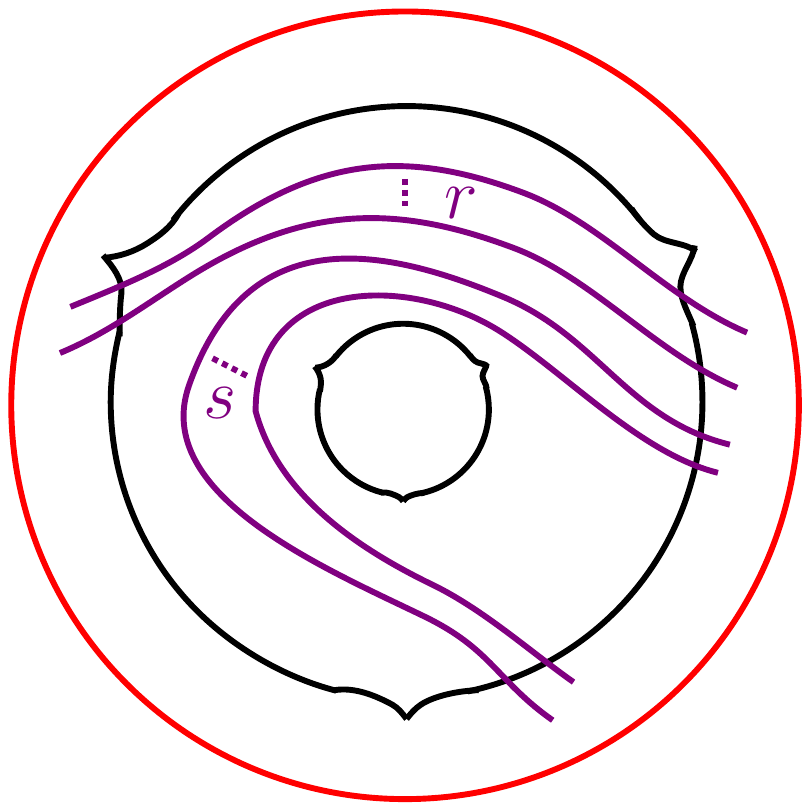}
\caption{The arcs corresponding to $\gamma_{cc}$ and $\gamma_{cb}$.}
\label{20200807Mirror}
\end{center}
\end{figure}

\section{Vertical $3$-manifolds determine $4$-manifolds}
%\section{Determine the $4$-manifold by vertical $3$-manifolds}
\begin{thm}\label{determining4Mfd}
Let $f:X\to\mathbf{R}^2$ be a simplified $(2, 0)$-trisection map. The $4$-manifold $X$ is determined as  follows:
\begin{itemize}
\item[(1)] The $4$-manifold $X$ with the $6$-tuple $\begin{pmatrix}
S^3 & S^3 & L((q - 1)^2, q - 1 + \epsilon)\\
S^1 \times S^2 & L(q - 2, \epsilon) & L(q, -\epsilon)
\end{pmatrix}$ is $S^2 \times S^2$ if $q$ is even and $\mathbf{CP}^2 \# \overline{\mathbf{CP}^2}$ if $q$ is odd and $q \neq 1$.

\item[(2)] The $4$-manifold $X$ with $\begin{pmatrix}
S^3 & L(9, 2\epsilon) & L(4, \epsilon)\\
L(2, 1) & L(5, \epsilon) & S^3
\end{pmatrix}$ is $\mathbf{CP}^2 \# \mathbf{CP}^2$ if $\epsilon =-1$ and $\overline{\mathbf{CP}^2} \# \overline{\mathbf{CP}^2}$ if $\epsilon =1$.
\item[(3)] The $4$-manifold $X$ with $\begin{pmatrix}
S^1 \times S^2 & L(4, 1) & L(4, 1)\\
S^3 & L(4 + \epsilon, 1) & S^3
\end{pmatrix}$ is  $\mathbf{CP}^2 \# \mathbf{CP}^2$ if $\epsilon=1$ and $\mathbf{CP}^2 \# \overline{\mathbf{CP}^2}$ if $\epsilon =-1$.%if $\delta = 1$ and $\mathbf{CP}^2 \# \overline{\mathbf{CP}^2}$ if $\epsilon = -1$.
\item[(4)] The $4$-manifold $X$ with $\begin{pmatrix}
S^1 \times S^2 & S^3 & S^3\\
S^3 & S^1 \times S^2 & S^3
\end{pmatrix}$ is $\mathbf{CP}^2 \# \overline{\mathbf{CP}^2}$.
\item[(5)] The $4$-manifold $X$ with $\begin{pmatrix}
S^1 \times S^2 & S^3 & S^3\\
S^3 & L(2, 1) & S^3
\end{pmatrix}$ is $\mathbf{CP}^2 \# \mathbf{CP}^2$ or $\overline{\mathbf{CP}^2} \# \overline{\mathbf{CP}^2}$.
\end{itemize}
\end{thm}
\begin{proof}
The 6-tuples in Cases (1) and (2) appear only in Case (A).
From the proof of Theorem~\ref{thmA}, we know
\[
[a_2] = \begin{pmatrix} 1 \\ 0 \end{pmatrix},\quad
[b_2] = \begin{pmatrix} 0 \\ 1 \end{pmatrix},\quad
[c'_2] = \begin{pmatrix} -1 \\ q \end{pmatrix},\quad
[a'_2] = \begin{pmatrix} p \\ \pm 1 \end{pmatrix}, 
\]where $p, q \in \mathbf{Z}$.

In Case (1), the $6$-tuple appears when either $V_{ba}$ or $V_{ac}$ is $S^1 \times S^2$. Assuming $V_{ba} = S^1 \times S^2$, we have $p = 0$ and $[a'_2] = \begin{pmatrix} 0 \\ \pm 1 \end{pmatrix}.$ An unwrinkle is a homotopy of a smooth map which replaces an innermost curve with outward three cusps in the singular value set into the image of a Lefschetz singularity, and a sink is a homotopy of a smooth map which replaces the image of a Lefschetz singularity and a part of the image of indefinite folds into a cusp, see \cite{lekili}.
We apply an unwrinkle to the inner, cusped circle on the image of $f$, move the image of the Lefschetz singularity obtained by the unwrinkle to the indefinite fold between $a_2$ and $a'_2$ and apply a sink as shown in Figure~\ref{sinkHomotopy}. 
Describe the simple closed curves $a_2, b_2, c'_2, a'_2$ on the standard torus embedded in $S^3$ and assign over/under information to their intersections in the order of the monodromy starting from $b_2$, that is $b_2, c'_2, a'_2, a_2$. See Figure~\ref{KirbyDiagramOf1stCase} for the case $q  > 0$. The framing coefficients of these simple closed curves are induced by the surface framings. This is a Kirby diagram of the $4$-manifold $X$. Canceling the two dotted circles and $a_2$ and $b_2$, we have a Hopf link with framing coefficients $\epsilon q$ and $0$, where $\epsilon$ is the sign in Theorem~\ref{thmA}. Remark that $\epsilon$ is needed since we replaced $q$ by $-q$ in the proof of Theorem~\ref{thmA} with $\epsilon =-1$. Hence $X$ is $S^2 \times S^2$ if $q$ is even and $X$ is $\mathbf{CP}^2 \# \overline{\mathbf{CP}^2}$ if $q$ is odd.
\begin{figure}[htbp]
\begin{center}
\includegraphics[width=9cm, bb=128 553 481 713]{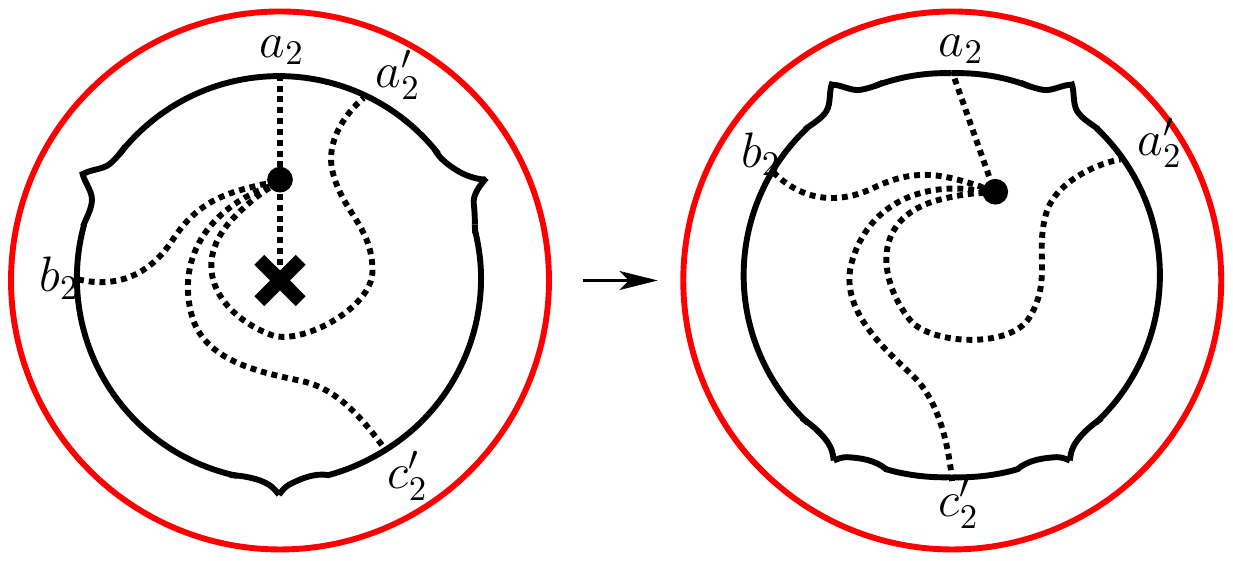}
\caption{A homotopy deformation by a sink.}
\label{sinkHomotopy}
\end{center}
\end{figure}
\begin{figure}[htbp]
\begin{center}
\includegraphics[clip, width=5.5cm, bb=197 493 397 713]{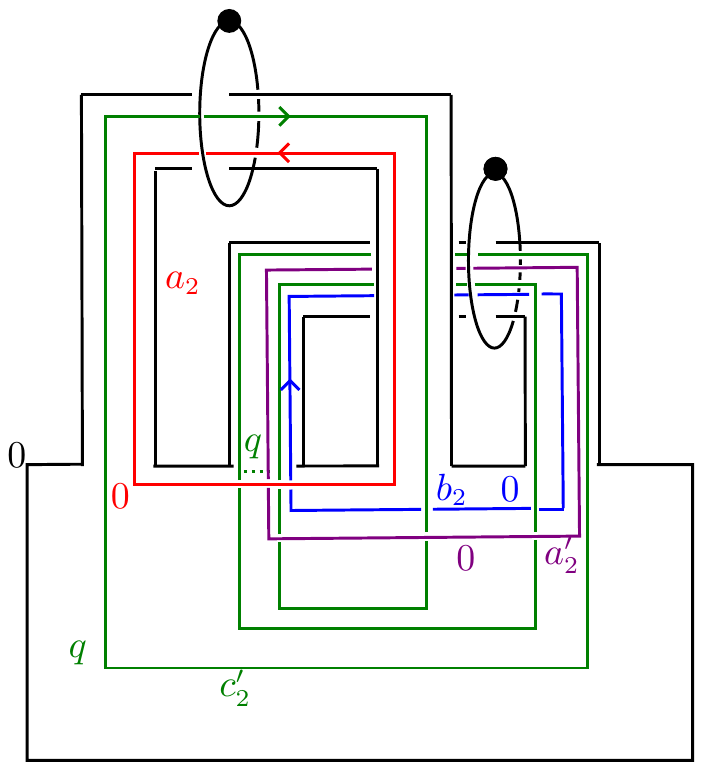}
\caption{The Kirby diagram in Case (1) with $q > 0$.}
\label{KirbyDiagramOf1stCase}
\end{center}
\end{figure}
In Case (2), the $6$-tuple appears when both of $V_{ba}$ and $V_{ac}$ are not $S^1 \times S^2$. Hence it is either (i) $p = -1$ and $[c'_2] = \begin{pmatrix} -1 \\ \pm 2 \end{pmatrix}$ or (ii) $p = -2$ and $[c'_2] = \begin{pmatrix} -1 \\ \pm 1 \end{pmatrix}$. We apply an unwrinkle and a sink as in the previous case. 
In Case (2)-(i), $a'_2 = \begin{pmatrix} -1 \\ \pm1 \end{pmatrix}$ and the Kirby diagram of $\overline{X}$ is as shown on the left in Figure~\ref{KirbyDiagramOf2ndCase} if $\mu = t_d$ and  on the right if $\mu = t^{-1}_d$. Here the $4$-manifold is the mirror image $\overline{X}$ of $X$, not $X$ itself, since we took the reflection at the end of the proof of Theorem \ref{thmA} with $p = -1$. Canceling the two dotted circles and $a_2$ and $b_2$, we have a Hopf link with framing coefficients $\pm1$ and $\pm2$. Therefore, $\overline{X}$ is $\mathbf{CP}^2 \# \mathbf{CP}^2$ if $\mu = t_d$ and it is $\overline{\mathbf{CP}^2} \# \overline{\mathbf{CP}^2}$ if $\mu = t^{-1}_d$. From the proof of Theorem~\ref{thmA}, we see that $\mu =t_d$ if  $\epsilon =1$ and $\mu =t^{-1}_d$ if $\epsilon =-1$. This coincides with the assertion in Case (2).% Here we used the reflection to adjust the obtained $6$-tuple to the one in the assertion. Hence $X$ is $\mathbf{CP}^2 \# \mathbf{CP}^2$.
\begin{figure}[H]
\begin{center}
\includegraphics[clip, width=12cm, bb=130 493 550 713]{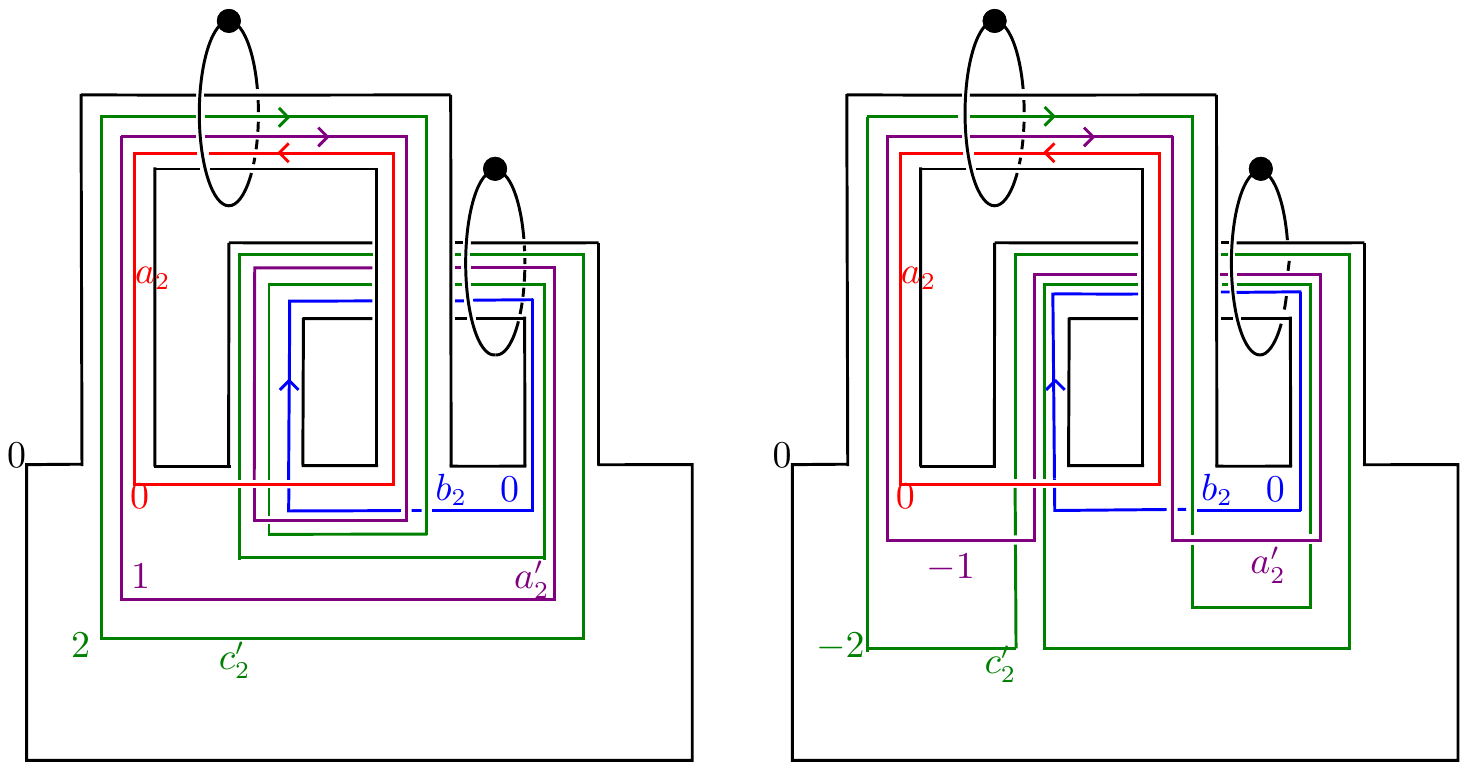}
\caption{The Kirby diagrams of $\overline{X}$ in Case (2)-(i).}
\label{KirbyDiagramOf2ndCase}
\end{center}
\end{figure}
In Case (2)-(ii), $a'_2 = \begin{pmatrix} -2 \\ \pm1 \end{pmatrix}$ and the Kirby diagram of $X$ is as shown on the left in Figure~\ref{KirbyDiagramOf2ndCase4} if $\mu = t_d$ and that on the right if $\mu = t^{-1}_d$. Applying the same argument as in Case (2)-(i) we have the same conclusion.

\begin{figure}[H]
\begin{center}
\includegraphics[clip, width=12cm, bb=130 493 555 713]{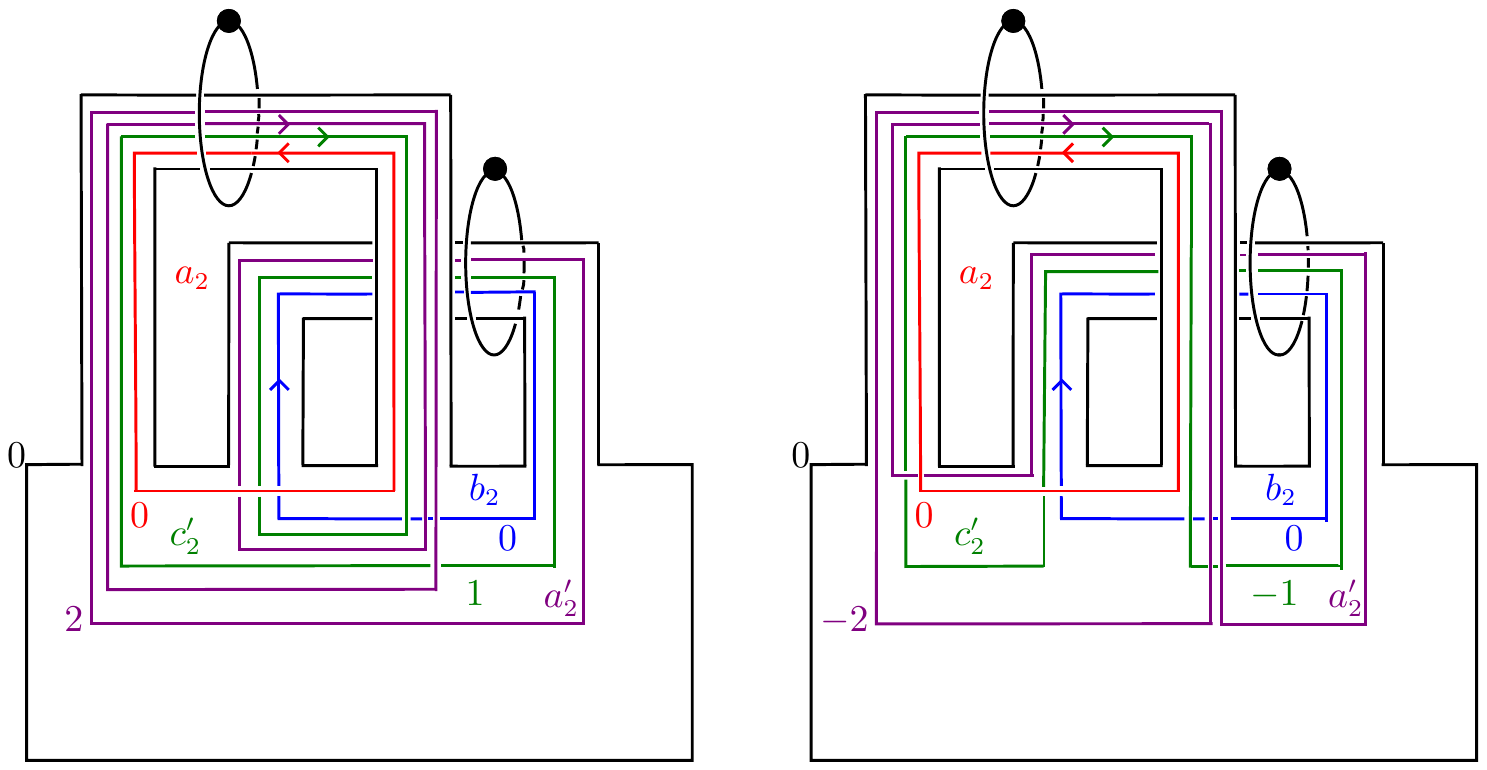}
\caption{The Kirby diagrams of $X$ in Case (2)-(ii).}
\label{KirbyDiagramOf2ndCase4}
\end{center}
\end{figure}
The $6$-tuple in Case (3) only appears in Case (B) with $\mu = t^{\pm4}_d$. We have%By equation~\eqref{proof2C_2}, $\pm\epsilon_2 = - 1$ and 
\begin{equation*}%\label{proofCaseB}
[a_2] = \begin{pmatrix} 1 \\ 0 \end{pmatrix},\quad
[b_2] = \begin{pmatrix} 0 \\ 1 \end{pmatrix},\quad
[c'_2] = \begin{pmatrix} -1 \\ \epsilon_2 \end{pmatrix},
\end{equation*} where $\epsilon_2 \in \{ -1, 1 \}$, and $\epsilon_2 = \mp \epsilon$ since $L(4\mp \epsilon_2, 1) = L(4 + \epsilon, 1)$.

Since it is in Case (B), the map can be deformed so that the image of indefinite folds consists of two circles with 3 cusps as shown on the right in Figure \ref{deformation}.
One of the $3$-cusped circles corresponds to the vanishing cycles $a_2, b_2, c'_2$ and it corresponds to a summand $\mathbf{CP}^2$ if $\epsilon_2 = -1$ and $\overline{\mathbf{CP}^2}$ if $\epsilon_2 = 1$. The other $3$-cusped circle corresponds to the inner $3$-cusped circle before the deformation and it corresponds to a summand $\mathbf{CP}^2$ if $\mu = t^{4}_d$ since the mutual positions of the three vanishing cycles of the inner $3$-cusped circle are those of $\mathbf{CP}^2$ and to a summand $\overline{\mathbf{CP}^2}$ if $\mu = t^{-4}_d$ since the mutual positions are those of $\overline{\mathbf{CP}^2}$ as shown in~\cite[Figure 6]{hayano}.
Suppose that $\mu = t^{4}_d$. If $\epsilon_2 = -1$, then $\epsilon =1$ and $X = \mathbf{CP}^2 \# \mathbf{CP}^2$. If $\epsilon_2 = 1$, then $\epsilon = -1$ and $X = \mathbf{CP}^2 \# \overline{\mathbf{CP}^2}$. 
Suppose that $\mu = t^{-4}_d$. Since $V_{bb} = L(\pm4, 1)$ in the proof of Theorem~\ref{thmB}, we need to apply the reflection so that the entry of $V_{bb}$ becomes $L(4, 1)$. Therefore, if $\epsilon_2 = -1$ then $\epsilon = -1$ and $X$ is the mirror image of $\mathbf{CP}^2 \# \overline{\mathbf{CP}^2}$, and  if $\epsilon_2 = 1$ then $\epsilon = 1$ and $X$ is the mirror image of $\overline{\mathbf{CP}^2} \# \overline{\mathbf{CP}^2}$. These coincide with the assertion in Case (3).
%Since $\epsilon =\mp\epsilon_2 =1$, $X$ is $\mathbf{CP}^2 \# \mathbf{CP}^2$ if $\epsilon_2 = -1$, where $\mu = t^4_d$, and $X$ is $\overline{\mathbf{CP}^2} \# \overline{\mathbf{CP}^2}$ if $\epsilon_2 = 1$, where $\mu = t^{-4}_d$.  This reverses the orientations of  $\overline{\mathbf{CP}^2} \# \overline{\mathbf{CP}^2}$. 

In Case (4), the $6$-tuple appears in Case (B) with $\mu = t^{\pm1}_d$ and $\pm \epsilon_2 = 1$. Hence $X$ is $\mathbf{CP}^2 \# \overline{\mathbf{CP}^2}$ by the same observation as in Case (3).

% If it is in Case (B),$a_2, b_2, c'_2$ are as in \eqref{proofCaseB} and hence the first summand of $X$ is $\mathbf{CP}^2$ if $\mu =t_d$ and $\overline{\mathbf{CP}^2}$ if $\mu =t^{-1}_d$, and the second summand is $\mathbf{CP}^2 $ if $\epsilon_2 = -1$ and $\overline{\mathbf{CP}^2}$ if $\epsilon_2 = 1$.

In Case (5), the $6$-tuple appears in Case (B) with $\mu = t^{\pm1}_d$ and $\pm \epsilon_2 =-1$. 
%In this case, $[c'_2] = \begin{pmatrix} -1 \\ \epsilon_2 \end{pmatrix}$. 
Hence $X$ is $\mathbf{CP}^2 \# \mathbf{CP}^2$ if $\epsilon_2 = -1$, where $\mu =t_d$, and $\overline{\mathbf{CP}^2} \# \overline{\mathbf{CP}^2}$ if $\epsilon_2 = 1$, where $\mu =t^{-1}_d$. This completes the proof.
\end{proof}
Corollary~\ref{determining4MfdC} stated in the introduction follows from Theorem~\ref{determining4Mfd} immediately.

%%%%%%%%%%%%%%%%%%%%%%%%%
%%%%%%%%%%%%%%%%%%%%%%%%%
\begin{comment}
{\large Memo : }handle slide generates the pair Not right-left eq.(reiwoireru)
kyouyakudehuhennmokaku

\begin{thm}
$\mathbf{CP}^2 \# \overline{\mathbf{CP}^2}$ admits two simplified $(2, 0)$-trisection maps with the same vertical manifolds but not to be right-left equivalent.
\end{thm}
\begin{proof}
The first simplified $(2, 0)$-trisection map is in Case (A) with $\mu = t_d$.
$[c'_2] = \begin{pmatrix} -1 \\ 1 \end{pmatrix}$,
$[a'_2] = \begin{pmatrix} 0 \\ 1 \end{pmatrix}$.
The second one is in Case (B) with $\mu = t_d$.
$[c'_2] = \begin{pmatrix} -1 \\ 1 \end{pmatrix}$,
$[a'_2] = \begin{pmatrix} 1 \\ 0 \end{pmatrix}$.
\end{proof}
\end{comment}
%%%%%%%%%%%%%%%%%%%%%%%%%%%
%%%%%%%%%%%%%%%%%%%%%%%%%%%
\section{A construction of vertical hyperbolic $3$-manifolds.}
%A stable map whose singular value set consists of concentric circles with cusps had been studied by Kobayashi in \cite{kobayashi}, where it is called `` asagao'', and Williams in \cite{williams}.
In this section, we give simplified $(2, 0)$-$4$-section maps that have hyperbolic vertical $3$-manifolds.%(cf. \cite{IN}).
\begin{dfn} 
A stable map $f : X \to \mathbf{R}^2$ is called a simplified $(g, k)$-$4$-section map if the following conditions hold:
\begin{itemize}
\item The singular value set of definite folds is a circle, bounding a disk $D$.
\item The singular value set of indefinite folds consists of $g$ concentric circles on $D$. Each of the inner $g-k$ circles has $4$ outward cusps.
\item The preimage of the point at the center is a closed orientable surface of genus $g$.
\end{itemize}
\end{dfn}
See Figure~\ref{2_4_gon_multi}.

\begin{thm}
Suppose $X = \#^2 S^2 \times S^2$ or $\#^2\mathbf{CP}^2 \#^2 \overline{\mathbf{CP}^2}$. Then there exists an infinite sequence of simplified $(2, 0)$-$4$-section maps $\{ f_i : X \to \mathbf{R}^2 \}_{i \in \mathbf{N}}$ 
such that their vertical $3$-manifolds %$\{ M_i\}$ 
over the arc $\omega$ in Figure~\ref{2_4_gon_multi} are hyperbolic and mutually non-diffeomorphic.% holding following :
%\begin{itemize}
%\item[(1)] the sigular value set of $f_i$ is as shown in Figure~\ref{2_4_gon_multi}
%\item[(2)] let $\omega$ be an arc as in Figure~\ref{2_4_gon_multi}, then $f^{-1}(\omega)$ is a hyperbolic $3$-manifold for all $i \in \mathbf{N}$. In particular, the surgery diagrams of these manifolds are $2$-bridge links.
%\item[(3)] if $i \neq j$, then $f_i^{-1}(\omega)$ are not diffeomorphic to $f_j^{-1}(\omega)$.
%\end{itemize}

\end{thm}
%\begin{prop}\label{hyperbolic}
%Let $\Sigma_2$ be a closed orientable surface of genus $2$, $\alpha = \{ \alpha_1, \alpha_2 \}, \beta = \{ \beta_1, \beta_2 \}, \gamma = \{ \gamma_1, \gamma_2 \}$ be collections of simple closed curves on $\Sigma_2$. Suppose that $(\alpha_1, \beta_1), (\beta_1, \gamma_1), (\alpha_2, \beta_2)$, and $(\beta_2, \gamma_2)$ intersect one point transversally on $\Sigma_2$, and $\alpha_1 \cap \alpha_2, \beta_1 \cap \beta_2, \gamma_1 \cap \gamma_2, \alpha_1 \cap \beta_2, \alpha_2 \cap \beta_1, \gamma_1 \cap \beta_2$, and $\gamma_2 \cap \beta_1$ are empty sets. Then there exists a $4$-manifold $W$ with boundary and a stable map $f : W \to \mathbf{R}^2$ that the singular value set is as shown in Figure~\ref{2_hone_sing}.
%\end{prop}

\begin{proof}
Let $\Sigma_2$ be a closed oriented surface of genus $2$ and describe simple closed curves $\alpha_1, \beta_1, \gamma_1, \alpha_2, \beta_2, \gamma_2$ on $\Sigma_2$ as in Figure~\ref{hyp_diagram}. The union of simple closed curves $\gamma_1, \gamma_2$ constitutes a $2$-bridge link with slope $\frac{p}{q}$ for $q$ being even if $\Sigma_2$ is assumed to be embedded in $S^3$ in the standard position. The curve $\gamma_1$ winds $l$ times around the left handle of $\Sigma_2$ and the curve $\gamma_2$ winds $r$ times around the right one.
\begin{figure}[H]
\begin{center}
\includegraphics[width=9cm, bb=205 625 393 708]{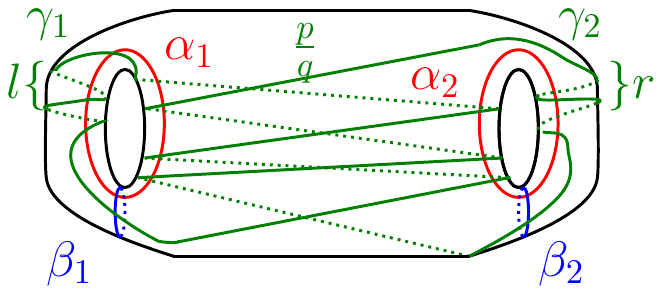}
\caption{The genus-$2$ surface $\Sigma_2$ embedded in $S^3$.}
\label{hyp_diagram}
\end{center}
\end{figure}
We will make a smooth map whose singular value set is the left half of the image of a simplified $(2, 0)$-$4$-section map as in Figure~\ref{2_hone_sing}. Identify the fiber over $p$ with $\Sigma_2$ and attach $2$-handles along  $\alpha_1, \beta_1, \gamma_1$ on $\Sigma_2$, so that we extend the stable map over a small neighborhood of $p$ to a neighborhood of the union of the reference paths for  $\alpha_1, \beta_1, \gamma_1$ from $p$. 
Since  $\alpha_1$ and  $\beta_1$ on $\Sigma_2$ intersect at one point transversely we may extend this map beyond the cusp. Since $\beta_1$ and $\gamma_1$ on $\Sigma_2$ intersect at one point transversely we can also extend the map beyond the other cusp.  We attach $2$-handles along $\alpha_2, \gamma_2$, and $\beta_2$. By Lemma \cite[Lemma 3.2]{hayano}, $\alpha'_2$ and $\alpha_2$ are isotopic and hence $\alpha'_2$ intersects $\beta_2$ once transversely.
% as \[\alpha_{2}' = t_{t_{\alpha_1}(\beta_1)}(\alpha_2) = \alpha_2 - (\alpha_2 \cdot t_{\alpha_1}(\beta_1))t_{\alpha_1}(\beta_1) = \alpha_2\]  Since
%\[ \gamma'_2 = t^{-1}_{t^{-1}_{\beta_1}(\gamma_1)}(\gamma_2)
%= \gamma_2 + (\gamma_2 \cdot t^{-1}_{\beta_1}(\gamma_1)) t^{-1}_{\beta_1}(\gamma_1)
%= \gamma_2\]
By the same reason, $\gamma'_2$ intersects $\beta_2$ once transversely. Hence we may extend the map beyond the two cusps on the outer, cusped half circle. By attaching a $3$-handle corresponding to a neighborhood of the image of definite folds, we get a smooth map  from a $4$-manifold $W$ with boundary whose vanishing cycles are as shown in Figure~\ref{hyp_diagram} and whose singular value set is as in Figure~\ref{2_hone_sing}.

\begin{figure}[htbp]
\begin{center}
\includegraphics[width=4.5cm, bb=212 499 373 713]{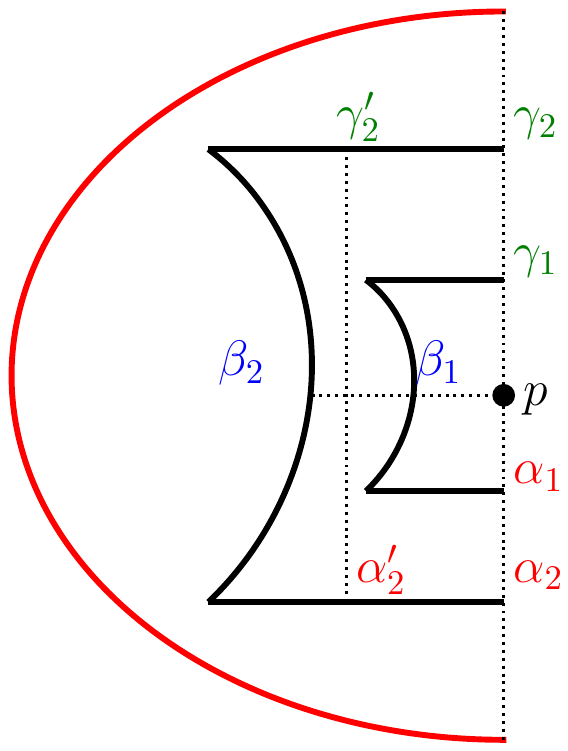}
\caption{A $4$-manifold $W$ with boundary and a stable map $f : W \to \mathbf{R}^2$.}
\label{2_hone_sing}
\end{center}
\end{figure}

The $4$-manifold $W$ has a handle decomposition into one $0$-handle $h^0$, two $2$-handles $h^2 \cup h^2$ attached along the vanishing cycles $\gamma_1, \gamma_2$ and one $4$-handle $h^4$ as shown in Figure~\ref{2_hone_sing_handle_1}. Therefore, the Kirby diagram of $W$ is given by the $2$-bridge link $\gamma_1 \cup \gamma_2$ with surface framing induced by $\Sigma_2$ in the standard position in $S^3$.

\begin{figure}[htbp]
\begin{center}
\includegraphics[width=12cm, bb=128 454 524 713]{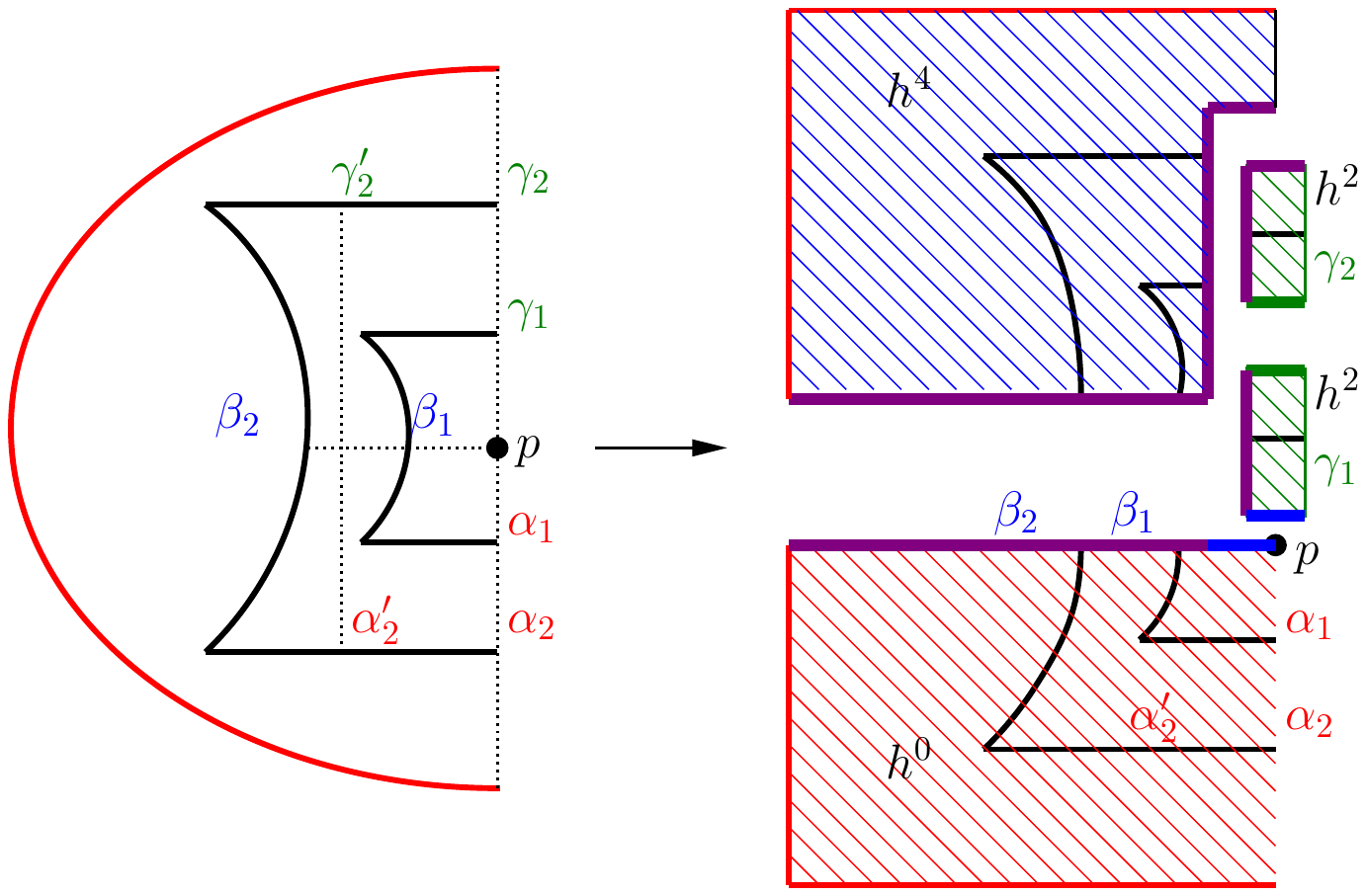}
\caption{The handle decomposition of $W$.}
\label{2_hone_sing_handle_1}
\end{center}
\end{figure}

By gluing this smooth map and its mirror we get a closed $4$-manifold $W \cup \overline{W}$, which is the double $DW$ of $W$, with a stable map from $DW$ to $\mathbf{R}^2$ whose singular value set is as shown in Figure~\ref{2_4_gon_multi}. The Kirby diagram of $DW$ is obtained from the Kirby diagram of $W$, which is the $2$-bridge link $\gamma_1 \cup \gamma_2$ with surface framing, by adding a meridional curve for each of the link components $\gamma_1$ and $\gamma_2$ with framing $0$. By Kirby calculus, it becomes a disjoint union of two Hopf links one of whose link component has framing $0$. Since we can change even/odd of the framing of the other link component by changing $r$ and $l$, we obtain the $4$-manifolds in the assertion.

The $3$-manifold on the boundary of $W$, which is the preimage of $\omega$ in the assertion is obtained from $S^3$ by applying Dehn surgeries along $\gamma_1$ and $\gamma_2$. The $2$-bridge link $\gamma_1 \cup \gamma_2$ is hyperbolic unless it is a torus link by \cite{menasco}. Now we assume $|l|$ and $|r|$ are sufficiently large so that the absolute values of the surgery coefficients become sufficiently large. For such $l$ and $r$ the surgered $3$-manifolds are hyperbolic by \cite{thurston}. Moreover, for sufficiently large $(l, r)$ and $(l', r)$, the corresponding hyperbolic $3$-manifolds are not mutually diffeomorphic by Thurston's hyperbolic Dehn surgery theorem. This completes the proof.
\end{proof}

\begin{rem}
In the above proof, we can show the existence of mutually non-diffeomorphic vertical $3$-manifolds by checking the torsions of their homology groups as follows. Let $(\mu_1, \lambda_1)$ and $(\mu_2, \lambda_2)$ be meridian-longitude pairs of the link components $\gamma_1$ and $\gamma_2$, respectively. Then the first integral homology of the complement of the link $\gamma_1 \cup \gamma_2$ is isomorphic to $\mathbf{Z}\langle\mu_1\rangle\oplus\mathbf{Z}\langle\mu_2\rangle$. For each $k = 1, 2$, the image of the meridian of the solid torus by the Dehn filling for $\gamma_k$ is given by $r_k\mu_k+\lambda_k$ for $r_k \in \mathbf{Z}$. Set $n$ to be the linking number of  $\gamma_1$ and $\gamma_2$. Then $\lambda_1=n \mu_{2}$ and $\lambda_2=n \mu_{1}$. Hence the first integral homology of the surgered manifold $f^{-1}(\omega)$ can be written as
\[
    H_1(f^{-1}(\omega);\mathbf{Z})=\langle \mu_1, \mu_2 \mid
 r_1\mu_1+n\mu_2,\; r_2\mu_2+n\mu_1\rangle.
 \]
 Choosing an infinite sequence of $(r_1, r_2)$ suitably, we can show the existence of mutually non-diffeomorphic vertical $3$-manifolds.
 \end{rem}

\end{document}